\documentclass[12pt,a4paper]{amsart}

\usepackage{amssymb}
\input{prepictex}
\input{pictex}
\input{postpictex}

\newtheorem{thm}{Theorem}[section]

\newtheorem{lem}[thm]{Lemma}

\newtheorem{cor}[thm]{Corollary}

\theoremstyle{definition}

\newtheorem{defn}[thm]{Definition}

\newtheorem{remark}[thm]{Remark}

\newsymbol\kk 207C

\newcommand{\CC}{{\mathbb C}}

\newcommand{\FF}{{\mathbb F}}

\newcommand{\NN}{{\mathbb N}}

\newcommand{\QQ}{{\mathbb Q}}

\newcommand{\ZZ}{{\mathbb Z}}

\newcommand{\cA}{{\mathcal A}}

\newcommand{\cC}{{\mathcal C}}
\newcommand{\cD}{{\mathcal D}}

\newcommand{\cG}{{\mathcal G}}
\newcommand{\cH}{{\mathcal H}}

\newcommand{\cK}{{\mathcal K}}

\newcommand{\cS}{{\mathcal S}}
\newcommand{\cT}{{\mathcal T}}

\newcommand{\cV}{{\mathcal V}}

\def\pa{\makebox[.5em][l]
            {a\hspace{-.23em}\begin{picture}(0,0)
                             \put(0,0){\raisebox{-1.4ex}{`}}
                             \end{picture}
            }}%
\newcounter{picture}

\newcommand{\D}{{\Delta}}
\newcommand{\G}{{\Gamma}}

\newcommand{\Om}{{\Omega}}

\newcommand{\w}{{\omega}}

\newcommand{\aut}{{\text{\rm Aut}}}
\newcommand{\autD}{{\text{\rm Aut}(\D)}}

\newcommand{\gpd}{{\G'}}

\newcommand{\Pmn}{{P_{m,n}}}

\newcommand{\sh}{{\sigma}}
\newcommand{\talpha}{{\widetilde{\alpha}}}
\newcommand{\tA}{{\widetilde{A}}}
\newcommand{\tbeta}{{\widetilde{\beta}}}
\newcommand{\tf}{{\widetilde{f}}}
\newcommand{\tg}{{\widetilde{g}}}
\newcommand{\tgamma}{{\widetilde{\gamma}}}
\newcommand{\shaut}{{\text{\rm Aut}_{\text{\rm S}}(\D)}}
\newcommand{\tTD}{{\widetilde{T}_D}}

\newcommand{\Tzz}{{T_{z,z}}}
\newcommand{\vD}{{\cV_\D}}
\newcommand{\Wmn}{{W_{m,n}}}
\newcommand{\Wp}{{W_{p,0}}}

\newcommand{\st}{{\; |\;}}
\newcommand{\twohd}{{{}\leftrightarrow{}}}
\newcommand{\sX}{{s'}}

\begin{document}

\title{A Haagerup Inequality for $\tA_1\times\tA_1$ and $\tA_2$ Buildings}
\author{Jacqui Ramagge}
\address{Mathematics Department, University of Newcastle, Callaghan, NSW 2308,
Australia}
\email{jacqui@maths.newcastle.edu.au}
\author{Guyan Robertson }
\address{Mathematics Department, University of Newcastle, Callaghan, NSW
2308, Australia}
\email{guyan@maths.newcastle.edu.au}
\author{Tim Steger}
\address{Istituto Di Matematica e Fisica, Universit\`a degli Studi di
Sassari, Via Vienna 2, 07100 Sassari, Italia}
\email{steger@ssmain.uniss.it}
\thanks{This research was partly funded by University of Newcastle RMC Project
Grant number 45/290/501 and by the Italian MURST. The first two
authors would like to thank the University of Newcastle study leave
program for support during the later stages of this project and the
University of Sassari for support and hospitality.}

\begin{abstract}
Haagerup's inequality for convolvers on free groups may be interpreted
as a result on~$\tA_1$ buildings, i.e.~trees. Here are proved
analogous inequalities for discrete groups acting freely on the
vertices of~$\tA_1\times\tA_1$ and~$\tA_2$ buildings. The results
apply in particular to groups of type-rotating automorphisms acting
simply transitively on the vertices of such buildings.  These results
provide the first examples of higher rank groups with property~(RD).
\end{abstract}

\maketitle

\section{Introduction}

U.~Haagerup has given a beautiful and useful estimate for convolvers
on a free group \cite[Lemma 1.4]{Hag}.  Suppose~$\G$ is the free group
on a set of generators $N_+$ and let~$N$ consist of the generators
from~$N_+$ and their inverses.  Each $c\in\G$ can be written uniquely
as $c=a_1a_2\cdots a_n$ with $a_j\in N$ and $a_ja_{j+1}\neq 1$.  This
product is called the reduced word for $c$, and if the reduced word
has $n$~factors, we say that $c$~is a word of length~$n$.  Haagerup's
inequality applies to a function $g\in\ell^1(\G)$ supported on the
words of length~$n$ and it asserts that
\[
\| f*g \|_2 \leq (n+1) \|f\|_2\|g\|_2 .
\]
Denoting by~$\rho$ the right regular representation of~$\ell^1(\G)$
on~$\ell^2(\G)$, the inequality reads $\|\rho(g)\|\leq (n+1)\|g\|_2$.

In~\cite{Hag} Haagerup's inequality was used in the course of
establishing that the reduced $C^*$-algebra of a free group on
finitely many generators has the metric approximation property.  In
another development, it follows from Haagerup's inequality that
Jolissaint's property~(RD) holds for free groups.  Haagerup's
inequality was extended to word hyperbolic groups in~\cite{Jo,Harp},
proving that they too satisfy property~(RD). This is an ingredient in
the proof of the Novikov conjecture for word hyperbolic groups
\cite{CoM}.  An exposition of this last result may be found
in~\cite[Chapter~III.5]{Connes}. Also a recent paper of A. Nevo \cite{nev}
provides a new application of the Haagerup inequality to ergodic
theorems on groups.

Restating Haagerup's proof geometrically, we find that it gives a
result somewhat more general than was orginally stated.  Suppose
$\Gamma$ acts freely on the vertices of a tree, and let~$d(u,v)$
denote graph theoretic distance.  Fix a vertex $v_0$ and define the
length of $c\in\Gamma$ as
\[
|c|=d(v_0,cv_0) .
\]
If $g\in\ell^1(\G)$ is supported on elements of length~$n$, then
$\|\rho(g)\|\leq (n+1)\|g\|_2$.  To recover the original
inequality, consider the free group~$\G$ acting simply transitively on
its Cayley graph with respect to the generating set~$N$ and take $v_0$
to be $1\in\Gamma$.

Trees are~$\tA_1$ buildings. This paper generalizes Haagerup's inequality
to~$\tA_1\times\tA_1$ and $\tA_2$~buildings.  This is the
first such generalization to ``higher rank'' groups of either
Haagerup's inequality or of property~(RD).

Let $\D$~be an $\tA_1\times\tA_1$ or~$\tA_2$ building.
In~\S\ref{geometric intro} we define the {\bf shape},
$\sh(u,v)\in\NN\times\NN$ between any two vertices~$u,v\in\D$ which
essentially gives the dimensions of the convex hull of the two
vertices in the building sense.  The index of
\[
\shaut = \{ c\in\autD \st
   \text{$\sh(cu,cv)=\sh(u,v)$ for all $u$, $v\in\vD$} \}
\]
in~$\autD$ is~$1$ or~$2$.

Define a function in two variables
\begin{equation*}
p(m,n)=
\begin{cases}
(m+1)(n+1) & \text{($\tA_1\times\tA_1$ case)} \\
1/2(m+1)(n+1)(m+n+2)\sqrt{\max(m,n)+1} &
   \text{($\tA_2$ case).}
\end{cases}
\end{equation*}
Our main result is the following.

\begin{thm}\label{gp-cor}
Suppose~$\D$ is an $\tA_1\times\tA_1$ or~$\tA_2$ building
and~$\G\leq\shaut$ acts freely on the vertices of~$\D$.
Fix any vertex~$v_0\in\D$ and define a shape function~$\sh$ on~$\G$ by
\[
\sh(c)=\sh(v_0,cv_0).
\]
If $g\in\CC\G$ is supported on words of shape $(m,n)$ and~$f\in\ell^2(\G)$,
then
\[
\|f*g\|_2\leq p(m,n)\;\|f\|_2\;\|g\|_2.
\]
\end{thm}

\noindent
In the~$\tA_1\times\tA_1$ case our bound is optimal.
 From an analysis of radial functions, we conjecture
that the optimal value of~$p(m,n)$ in the~$\tA_2$ case
is~$1/2(m+1)(n+1)(m+n+2)$. Using a result in~\cite{S}, A.~Valette~\cite{V}
has shown that there is a Haagerup inequality applying to radial functions on
an~$\tA_n$ group.

Among buildings of dimension~$2$ only
$\tA_1\times\tA_1$ and $\tA_2$~buildings ever admit simply transitive
actions on their vertices. It was this that led us to consider these two
cases --- by happy coincidence they are also the easiest.
Indeed, Haagerup's original method, broadly conceived,
handles $\tA_1\times\tA_1$ buildings, and only one new ingredient is
required for the $\tA_2$~case.  It is natural to conjecture that
analogous versions of Haagerup's inequality hold for all types of
affine buildings. In fact, Valette has conjectured~\cite[p.~70]{FRR} that
any group acting properly and cocompactly on either a Riemannian symmetric
space or an affine building has property~(RD).

\subsection{Some analytical results used in the proofs}
We note the following well-known results which will be used on several
occasions.  Their proofs are straightforward and so are omitted.

\begin{lem}\label{B}
Let $T:\cH\longrightarrow\cK$ be an operator between two Hilbert
spaces, and suppose we have orthogonal decompositions
$\cH=\oplus_j\cH_j$ and $\cK=\oplus_k\cK_k$.  Expressing~$T$ as an
operator matrix $[T_{kj}]$ where $T_{kj}:\cH_j\longrightarrow\cK_k$,
one has
\[
\| T\| \leq \left\| \:\left[ \, \| T_{kj} \|\, \right] \:\right\|
\leq \left( \sum_{j,k}  \| T_{kj}\|^2 \right)^{1/2}
\]
where all norms are operator norms and
$\left[ \, \| T_{kj} \|\, \right]$ is the matrix with scalar entries
$\| T_{kj} \|$.
\end{lem}

\begin{lem}\label{C}
If $T=[t_{ij}]$, $S=[s_{ij}]$ with $0\leq t_{ij}\leq s_{ij}$
for all $i,j$, then $\| T\| \leq \| S\|$.
\end{lem}

\subsection{The buildings and the shape between a pair of vertices.}%
\label{geometric intro}

Given an~$\tA_n$ building~$\D$, there is a type map~$\tau$ defined on the
vertices of~$\D$ such that $\tau(v)\in\ZZ/(n+1)\ZZ$ for each vertex $v\in\D$.
A brief inductive argument shows that every~$a\in\autD$ gives rise to a
permutation of the set of types.
An automorphism~$a$ of~$\D$ is said to be {\bf type-rotating} if
there exists $i\in\{0,1,\ldots,n\}$ such that $\tau(av)=\tau(v)+i$ for all
vertices $v\in\D$. More generally, suppose that~$\D$ is an
$\tA_{n_1}\times\cdots\times\tA_{n_k}$ building.
There is a type map~$\tau$ on the vertices of~$\D$ where
\[
\tau(v)=
(\tau(v)_1,\ldots,\tau(v)_k)\in\ZZ/(n_1+1)\ZZ\times\cdots\times\ZZ/(n_k+1)\ZZ
\]
is a $k$-tuple.  Again, any automorphism~$a$ of~$\D$ gives rise to a
permutation on the set of types.  We say such an~$a$ is {\bf
type-rotating} if there exists a $k$-tuple
$(i_1,\ldots,i_k)\in\{0,1,\ldots,n_1\}\times\cdots\times\{0,1,\ldots,n_k\}$
such that
\[
\tau(av)=(\tau(v)_1+i_1,\ldots,\tau(v)_k+i_k).
\]
It follows (see~\S\ref{simply transitive A1}) that~$a$ is a
type-rotating automorphism of~$\D$ if and only if it is a Cartesian
product of type-rotating automorphisms of the~$k$ factors.
Recall that an~$\tA_1$ building is a tree and note that any
automorphism of an~$\tA_1$ building is type-rotating.

Henceforth, let $\D$ be an~$\tA_1\times\tA_1$ or~$\tA_2$ building.
The subgroup of type-rotating automorphisms will then be of index at
most~$2$ in $\autD$. This follows from the facts that $[D_8:C_2\times C_2]
= [D_6:C_3] =2$, where $C_n$ is the cyclic group and $D_n$ is the
dihedral group of order $n$.  An {\bf apartment} in~$\D$ is a chamber
subcomplex of~$\D$ isomorphic to the corresponding Coxeter
complex. Thus an apartment in~$\D$ is a plane tessellated by squares
in the $\tA_1\times\tA_1$ case and by equilateral triangles in
the~$\tA_2$ case.

Denote by~$\vD$ the vertex set of~$\D$. Any two vertices~$u,v\in\vD$ belong
to a common apartment. The convex hull, in the sense of buildings, between
two vertices $u$ and $v$ is depicted in Figure~\ref{convex hulls}, with
various degeneracies possible.
\refstepcounter{picture}
\begin{figure}[htbp]
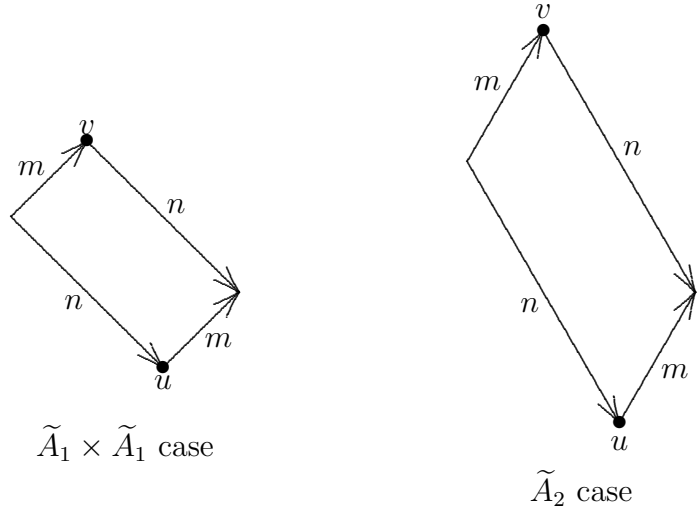
\label{convex hulls}
\hfil
\beginpicture
\setcoordinatesystem units <0.5cm,0.866cm>  point at -6 0  % sets scale
\setplotarea x from -5 to 5, y from -2.5 to 4.5            % sets plot size up
\put{$m$}  [tl]  at  2.1 -1.1
\put{$m$}  [br]  at -2.1  3.1
\put{$n$}  [tr]  at -1.1 -0.1
\put{$n$}  [bl]  at  1.1  2.1
\put{$\bullet$}  at  1   -2
\put{$u$}  [t]   at  1   -2.2
\put{$\bullet$}  at -1    4
\put{$v$}  [b]   at -1    4.2
\arrow <10pt> [.3,.67] from   1    -2   to     3   0
\arrow <10pt> [.3,.67] from  -3     2   to    -1   4
\arrow <10pt> [.3,.67] from  -1     4   to     3   0
\arrow <10pt> [.3,.67] from  -3     2   to     1  -2
\put {$\tA_2$ case}   at   0 -3
\setcoordinatesystem units <0.5cm,0.5cm> point at 6 0   % sets scale
\setplotarea x from -5 to 5, y from -2.5 to 4.5         % sets plot size up
\put{$m$} [tl]   at   2.1 -1.1
\put{$m$} [br]  at -2.1   3.1
\put{$n$}  [tr]   at -1.1 -0.1
\put{$n$}  [bl]   at   1.1   2.1
\put{$\bullet$} at    1   -2
\put{$u$}  [t]    at    1    -2.2
\put{$\bullet$} at   -1   4
\put{$v$}  [b]    at -1      4.2
\arrow <10pt> [.3,.67] from   1    -2   to     3   0
\arrow <10pt> [.3,.67] from  -3     2   to    -1   4
\arrow <10pt> [.3,.67] from  -1     4   to     3   0
\arrow <10pt> [.3,.67] from  -3     2   to     1  -2
\put {$\tA_1\times\tA_1$ case}   at   0 -4
\endpicture
\hfil
\caption{Convex hull of two vertices.}
\end{figure}
Fundamental properties of buildings imply that any apartment containing~$u$
and~$v$ must also contain their convex hull.

Define the {\bf distance}, $d(u,v)$, between $u$ and $v$ to be the
graph theoretic distance on the one-skeleton of~$\D$. Any path from
$u$ to $v$ of length $d(u,v)$ lies in their convex hull, and the union
of the vertices in such paths is exactly the set of vertices in the
convex hull.  Note that although this last statement is true
for~$\tA_1\times\tA_1$ and $\tA_2$~buildings, it does not hold for
arbitrary affine buildings. For example in a $\widetilde G_2$~building
it fails for two vertices whose convex hull contains at least three chambers.

We define the {\bf shape}~$\sh(u,v)$ of the ordered pair of vertices
$(u,v)\in\vD\times\vD$ to be the pair $(m,n)\in\NN\times\NN$ as
indicated in Figure~\ref{convex hulls}.  Note that $d(u,v)=m+n$.  In
the $\tA_2$ case the arrows in Figure~\ref{convex hulls} point in the
direction of cyclically increasing type, i.e.\
$\{\ldots,0,1,2,0,1,\ldots\}$.  In the $\tA_1\times\tA_1$ case the
components of $\sh(u,v)$ indicate the relative contributions to
$d(u,v)$ from the two $\tA_1$ factors.  More specifically, suppose
$\D=\D_1\times\D_2$ where $\D_1$ and $\D_2$ are $\tA_1$ buildings,
i.e.\ trees.  If $v=(v_1,v_2)$ and $w=(w_1,w_2)$ are vertices of $\D$,
the {\bf shape} from $v$ to $w$ is
\[
\sh(v,w)=\left( d(v_1,w_1), d(v_2,w_2)\right)
\]
where $d$ denotes the usual graph-theoretic distance on a tree.
An edge in $\D$ connects the vertices $v$ and $w$ if $\sh(v,w)=(0,1)$ or
$\sh(v,w)=(1,0)$.

\begin{lem}\label{shape decomposition for vertices}
Suppose $m_1,m_2,n_1,n_2\in\NN$ and $\sh(u,w)=(m_1+m_2,n_1+n_2)$ for
vertices $u,w\in\vD$. Then there is a unique vertex $v\in\vD$ such that
\[
\sh(u,v)=(m_1,n_1)\quad\text{ and }\quad \sh(v,w)=(m_2,n_2).
\]
\end{lem}
\begin{proof}
Such a ~$v\in\vD$ satisfies $d(u,w)=d(u,v)+d(v,w)$ so it must lie in the
convex hull of~$u$ and~$w$. Inside the convex hull existence and uniqueness
of~$v$ are clear.
\end{proof}

It is a direct consequence of the definitions that every type-rotating
automorphism~$a\in\autD$ preserves shape in the sense that
$\sh(au,av)=\sh(u,v)$ for all~$u,v\in\vD$. Furthermore, if~$b\in\autD$
is not type-rotating then $\sh(bu,bv)=(n,m)$ whenever
$\sh(u,v)=(m,n)$. Thus an element~$a\in\autD$ preserves shape if and
only it is type-rotating. Thus the group $\shaut$ of shape-preserving
automorphisms coincides with the group of type-rotating automorphisms.

We refer to Section \ref{simply transitive actions}
 for a more detailed analysis of the interesting
special case where the group~$\G\leq\shaut$ acts simply transitively on~$\vD$.

\subsection{Free actions and groupoids.}\label{free actions}

Suppose now that~$\G\leq\shaut$ acts freely, but not necessarily
transitively, on the vertices~$\vD$ of~$\D$.
This induces an action of~$\G$ on~$\vD\times\vD$.
Denote by $[u,v]$ the orbit of the ordered pair $(u,v)\in\vD\times\vD$
under the left action of~$\G$. Define
\[
\gpd = \G\backslash(\vD\times\vD),
\]
the set of $\G$-orbits of~$\vD\times\vD$.
Given $[u,v]$ and $[v,w]\in\gpd$, define their product as
\[
[u,v][v,w]=[u,w].
\]
The product of $[u,v]$ and $[x,y]$ is not always defined. It exists only
if~$v$ and~$x$ lie in the same~$\G$-orbit, and in that case it is well
defined because~$\G$ acts freely on~$\vD$. Specifically, if~$bx=v$
for~$b\in\G$, then
\[
[u,v][x,y]=[u,v][bx,by]=[u,by].
\]
The set~$\gpd$ with this product satisfies the axioms of a groupoid,
see~\cite[Chapter~II.5]{Connes}:

\begin{defn}
 A {\bf groupoid} consists of a set $\cG$, a distinguished subset
$\cG^{(0)} \subset \cG$,
two maps $s,r: \cG \longrightarrow \cG^{(0)}$ and a law of composition
$(\alpha_1,\alpha_2) \in \cG^{(2)} \mapsto  \alpha_1\alpha_2 \in \cG$, with
domain
$\cG^{(2)} = \{(\alpha_1,\alpha_2) \in \cG \times \cG : s(\alpha_1) =
r(\alpha_2)\}$,
such that
\begin{itemize}
\item $s(\alpha_1\alpha_2) = s(\alpha_2)$, $r(\alpha_1\alpha_2) =
r(\alpha_1)$, for all $(\alpha_1,\alpha_2) \in \cG^{(2)}$;
\item $s(\alpha) = r(\alpha) = \alpha$ , for all $\alpha \in \cG^{(0)}$;
\item $\alpha s(\alpha) = \alpha$ , $r(\alpha)\alpha = \alpha$, for all $\alpha \in \cG$;
\item $(\alpha_1\alpha_2)\alpha_3 = \alpha_1(\alpha_2\alpha_3)$;
\item each $\alpha$ has a two-sided inverse $\alpha^{-1}$, with
$\alpha\alpha^{-1} = r(\alpha)$,
   $\alpha^{-1}\alpha = s(\alpha)$.
\end{itemize}
\end{defn}
The maps $r$ and $s$ are called the {\bf range} and {\bf source}
maps. An element $\iota \in \cG^{(0)}$ is called a {\bf unit}.
The units in~$\gpd$ are of the form~$[v,v]$ for~$v\in\vD$. If
$\alpha=[u,v]\in\gpd$ then $r(\alpha)=[u,u]$, $s(\alpha)=[v,v]$
and~$\alpha^{-1}=[v,u]$.

\begin{defn}
Let~$\cG$ be a groupoid and~$X$ be a set together with a surjection
$\sX : X \to \cG^{(0)}$.
Form the fibred product
\begin{equation*}
X*\cG = \{(v,\alpha) \in X \times \cG\ \vert \ \sX(v)=r(\alpha)\}
\end{equation*}
A (right) {\bf groupoid action} of~$\cG$ on~$X$ is a map
$(v,\alpha)\mapsto v\alpha\in X$ from $X*\cG$ to $X$ such that
\begin{itemize}
\item $\sX(v\alpha)=s(\alpha)$;
\item $v(\alpha\beta)= (v\alpha)\beta$, whenever $(v,\alpha)\in X*\cG$
and $(\alpha,\beta)\in \cG^{(2)}$;
\item $v\sX(v)=v$ for all $v \in X$.
\end{itemize}
We think of
$\sX$ as being a `generalized source' map for the action.
%We trust that there will be no confusion between this (local) use of $\sh$
%and its use elsewhere to denote `shape' in buildings.
A groupoid action is called {\bf simply transitive} if in addition
\begin{itemize}
\item given $v,w\in X$ there exists a unique $\alpha\in\cG$ such that
$(v,\alpha) \in X*\cG$ and $v\alpha=w$.
\end{itemize}
\end{defn}

We refer the reader to~\cite{ren} and~\cite[Section~2]{MW} for details on
groupoid
actions.
When~$\cG$ acts simply transitively on~$X$ and~$v,w\in X$, we write $v^{-1}w$
for the unique $\alpha\in\cG$ such that $v\alpha=w$. This notation is possibly
misleading since $v^{-1}$ has no independent existence. However, if
$(v,\beta)$, $(w,\gamma) \in X*\cG$ then
\[
(v\beta)^{-1}(w\gamma)= \beta^{-1}(v^{-1}w)\gamma
\]
as is easily checked.

We define a simply transitive groupoid action of~$\gpd$ on~$\vD$ via
\[
v[v,w]=w.
\]
Thus $\sX(v)=[v,v]$ and 
$(u,[v,w]) \in X*\cG$  exactly when~$u$ and~$v$ lie in the
same $\G$-orbit.
If 
$v,w\in\vD$ then $v\alpha=w$ if and only if $\alpha=[v,w]$
so that~$v^{-1}w=[v,w]$.

\begin{defn}
Let~$\D$ be an~$\tA_1\times\tA_1$ or~$\tA_2$ building, $\vD$ its vertex set
and $\sh_\D$ its shape function. Let~$\cG$ be a groupoid acting on~$\vD$ on
the right. A {\bf shape function} on~$\cG$ compatible with the action is a map
$\sh_\cG:\cG\to\NN\times\NN$ such that
\[
\sh_\D(v,v\alpha)=\sh_\cG(\alpha)
\]
whenever $v\in\vD,\alpha\in\cG$ and $v\alpha$ is defined.
\end{defn}

Henceforth we will omit the subscripts and denote all shape functions by~$\sh$.
A shape function on $\gpd$ compatible with the action of~$\gpd$
on~$\D$ is given by
\[
\sh([u,v])=\sh(u,v)
\]
for any $u,v\in\vD$. This is well-defined since~$\G\leq\shaut$ acts
on~$\D$ by shape-preserving automorphisms.

\begin{lem}\label{shape decomposition for groupoid}
Let~$\D$ be an $\tA_1\times\tA_1$ or~$\tA_2$ building. Let~$\cG$ be a
groupoid endowed with a simply transitive action on~$\vD$ and a
compatible shape function. Let~$m_1,m_2,n_1,n_2\in\NN$. Whenever
$\alpha\in\cG$ has shape $(m_1+m_2,n_1+n_2)$ there exist unique
$\beta,\gamma\in\cG$ such that
\[
\alpha=\beta\gamma, \quad\quad \sh(\beta)=(m_1,n_1)
\quad\text{ and }\quad \sh(\gamma)=(m_2,n_2).
\]
\end{lem}
\begin{proof}
Fix $u\in\vD$ such that $w=u\alpha$ exists. A pair $(\beta,\gamma)$
satisfies the required conditions if and only if
\[
\sh(\beta)=(m_1,n_1), \quad\quad \sh(\beta^{-1}\alpha)=(m_2,n_2)
\quad\text{ and }\quad \gamma=\beta^{-1}\alpha.
\]
This happens if and only if~$v=u\beta$ satisfies
\[
\sh(u,v)=(m_1,n_1),\quad\text{ and }\quad\sh(v,w)=(m_2,n_2).
\]
By Lemma~\ref{shape decomposition for vertices}, there is exactly one
such vertex~$v$. Hence the pair $(\beta,\gamma)=(u^{-1}v,v^{-1}w)$
uniquely satisfies the required conditions.
\end{proof}

To summarize, any group~$\G\leq\shaut$ with a free left action on~$\vD$ gives
rise to a groupoid~$\gpd$ with a simply transitive (right) groupoid action
on~$\vD$ and a compatible shape function. The actions of~$\G$
and~$\gpd$ commute in the sense that
\[
c(v\alpha)=(cv)\alpha
\]
for $c\in\G$, $v\in\vD$ and $\alpha\in\gpd$ provided one side or the
other exist. We will call~$\gpd$ the {\bf commutant groupoid} of~$\G$.
Note that the left action of~$\G$ on~$\vD$ preserves the building
structure and that the right action of~$\gpd$ on~$\vD$ does not.

If, as in~\S\ref{simply transitive actions},~$\G$ acts simply transitively
on~$\vD$, then~$\gpd$ is a group isomorphic to~$\G$ and the groupoid action
of~$\gpd$ on~$\vD$ constructed in the above manner is equivalent to a
right group action of~$\G$ on~$\vD$. For any fixed~$v_0\in\vD$, the map
$c\mapsto[v_0,cv_0]$ gives an isomorphism $\G\to\gpd$. This isomorphism
depends on the choice of~$v_0$, varying up to inner automorphisms of~$\G$.
Our results will be phrased in the language of groupoid actions
with implications for free actions of groups on~$\vD$ via the above
construction.

\subsection{Convolutions}\label{convolutions}

If $\cG$~is a groupoid, let $\CC\cG$ denote the space of
finitely supported complex valued functions on~$\cG$.  The
convolution product on~$\CC\cG$ defined by
\[
(g_1*g_2)(\gamma)=\sum_{\alpha\beta=\gamma} g_1(\alpha)g_2(\beta)
\]
makes~$\CC\cG$ into an associative algebra.  Suppose that
$\cG$~acts simply transitively on a set~$X$.  For $g\in\CC\cG$,
define $\rho(g)\colon\ell^2(X)\to\ell^2(X)$ by
\[
(\rho(g)f)(v)=(f*g)(v)=
\sum_{\substack{u\in X \\ u\alpha=v}} f(u)g(\alpha) .
\]
It follows from the simple transitivity of the action
of~$\cG$ that~$\rho(g)f$ actually lies in~$\ell^2(X)$; indeed that
$\|f*g\|_2\leq \|f\|_2 \|g\|_1$.  The map~$\rho$ defines a right
action of~$\CC\cG$ on~$\ell^2(X)$.

Let $\D$~be an $\tA_1\times\tA_1$ or~$\tA_2$ building and let
\begin{equation}\label{polys}
p(m,n)=
\begin{cases}
(m+1)(n+1) & \text{($\tA_1\times\tA_1$ case)} \\
1/2(m+1)(n+1)(m+n+2)\sqrt{\max(m,n)+1} &
   \text{($\tA_2$ case).}
\end{cases}
\end{equation}
In Sections~\ref{A1 case} and~\ref{A2 case} we prove the
following result:

\begin{thm}\label{gpd-thm}
Suppose $\cG$~is a groupoid acting simply
transitively on~$\vD$ and $\sh\colon\cG\to\NN$ is a shape
function compatible with the action.  Fix any
$(m,n)\in\NN\times\NN$.  If $g\in\CC\cG$ is supported on
elements of shape~$(m,n)$, then $\|\rho(g)\|\leq p(m,n)\|g\|_2$.
\end{thm}

\noindent Equivalently, the conclusion asserts that
\[
\|f*g\|_2 \leq p(m,n) \|f\|_2\|g\|_2 .
\]
We can and will restrict to nonnegative and finitely supported~$f$
and~$g$ when proving this.

If $\G$ is a discrete group, we define convolution between
$f\in\ell^2(\G)$ and~$g\in\CC\G$ in the usual way:
\[
(f*g)(c)=\sum_{ab=c} f(a)g(b) .
\]
Theorem~\ref{gp-cor} is a consequence of Theorem~\ref{gpd-thm}.

\subsection*{Proof of Theorem~\ref{gp-cor}}
Let $\gpd$~be the commutant groupoid of~$\G$, as defined
in~\S\ref{free actions}.
Given~$f\in\ell^2(\G)$ and $g\in\CC\G$,
define $f'\in\ell^2(\vD)$ and $g'\in\CC\gpd$ by
\[
\begin{array}{rcll}
f'(cv_0) &=& f(c) &\text{ for $c\in\G$} \\
f'(w) &=& 0 &\text{ unless $w\in\G v_0$} \\
g'([dv_0,cv_0]) &=& g(d^{-1}c) &\text{ for $c$, $d\in\G$}\\
g'([w_1,w_2]) &=& 0 &\text{ unless $w_1$, $w_2\in\G v_0$.}
\end{array}
\]
It is then immediate that $f'*g'$ is related to $f*g$ via
\[
\begin{array}{rcll}
(f'*g')(cv_0) &=& (f*g)(c) &\text{ for $c\in\G$} \\
(f'*g')(w) &=& 0 &\text{ unless $w\in\G v_0.$}
\end{array}
\]
Moreover, $g'$ is supported on elements of~$\gpd$ of shape~$(m,n)$,
$\|g'\|_2=\|g\|_2$, and $\|f'\|_2=\|f\|_2$. Now apply
Theorem~\ref{gpd-thm} to the groupoid~$\gpd$ and the functions~$f'$
and~$g'$ to complete the proof of Theorem~\ref{gp-cor}. \hfill $\qedsymbol$
\vspace{4mm}

Recall from~\cite{Jo} that a group~$G$ has {\bf property~(RD)} if
there is a length function~$L$ on~$G$ such that any function on~$G$
which is rapidly decreasing relative to~$L$ belongs to the
reduced~$C^*$-algebra of~$G$.

\begin{cor}\label{RD-cor}
Any group $\G\leq\autD$ which acts freely on~$\vD$ has property~(RD).
\end{cor}
\begin{proof}
The index of $\shaut$ in $\autD$ is at most~$2$.  Hence the index of
$\G\cap\shaut$ in $\G$ is at most~$2$.  As property~(RD) holds for a
group whenever it holds for some subgroup of finite
index~\cite[Proposition~2.1.5]{Jo}, we may assume that $\G\leq\shaut$.

Fix $v_0\in\vD$ and define $\sh(c)=\sh(v_0,cv_0)$ for~$c\in\G$.
Likewise, define
\[
|c|=d(v_0,cv_0) .
\]
Then $|c|$ is a length function on~$\G$, in the sense of~\cite{Jo}.
When $\sh(c)=(m,n)$ one has $|c|=m+n$.  To derive property~(RD) from
Theorem~\ref{gp-cor} we note that a given value of
$|c|$ corresponds to only $|c|+1$ different shapes, and that the
relevant function~$p(m,n)$ from~(\ref{polys}) can be bounded by a
polynomial in~$m+n$. Now argue as in~\cite[Chapter~III.5, Theorem~5]{Connes}
or~\cite[Lemma~1.5]{Hag}.
\end{proof}

\begin{remark}
Corollary~\ref{RD-cor} implies property~(RD) for any discrete subgroup
of $\autD$ which has a torsion-free subgroup of finite index. As a consequence,
if $\FF$ is a finite extension of $\QQ_p$ and $\Gamma$ is a finitely
generated discrete
subgroup of ${\text{\rm{SL}}}_3(\FF)$, then it follows from Selberg's Lemma
\cite{sel} that $\Gamma$ has a torsion free finite index subgroup,
and so satisfies property (RD).
\end{remark}

\subsection{The retraction centred at a boundary point of the building}

The proof of Theorem~\ref{gpd-thm} relies heavily on the notion of the
retraction of a building centred at a boundary point. We outline a
construction based on~\cite[page 171]{Brown}.  There is also a concise
description in~\cite[\S9.3]{Ronan}.  An infinite straight line of
edges in an apartment is called a {\bf wall}.  Fix an apartment~$\cA$
of~$\D$. Consider two half-planes in~$\cA$ bounded by intersecting
walls and measure the angle between those two walls through the
intersection of the two half-planes.  If that angle is nonzero and
minimal, we call the intersection of the two half-planes a sector.  A
{\bf sector} in~$\D$ is by definition a sector in some apartment
of~$\D$. The boundary,~$\Om$, of~$\D$ is the set of sectors in $\D$
under the equivalence relation that $\cS\sim\cS'$ if $\cS\cap\cS'$
contains a sector. (This is the analogue in higher rank of the notion
of tail equivalence for semi-infinite paths in trees.)  Given a vertex
$v\in\vD$ and a boundary point $\w\in\Om$ there is a unique sector
based at~$v$ representing~$\w$, denoted~$\cS_v(\w)$.  Thus~$\Om$ can
be identified with the set of sectors emanating from any fixed vertex
in~$\D$.

Fix an apartment $\cA_0$ of~$\D$ and a sector~$\cS_0$ in~$\cA_0$ representing a
boundary point $\w_0\in\Om$. Then~$\D$ is a union of apartments~$\cA'$ each of
which contains a subsector of~$\cS_0$. There is a retraction
$r:\D\longrightarrow \cA_0$ which is a contraction, in the sense that
$d(r(u),r(v))\leq d(u,v)$ for any $u,v\in\vD$, and whose restriction to any
such~$\cA'$ is an isomorphism from~$\cA'$ onto~$\cA_0$ which fixes
$\cA'\cap\cA_0$ pointwise. In particular, given any $v\in\vD$, $r(v)$ is
independent of the apartment~$\cA'$ chosen such that $v\in\cA'$.

The boundary point~$\w_0$ defines a preferred direction in~$\cA_0$ which we
shall call {\bf up}. Each chamber in $\cA_0$ may therefore be labelled with
an arrow pointing up, as in Figure~\ref{arrows up}.
\refstepcounter{picture}
\begin{figure}[htbp]
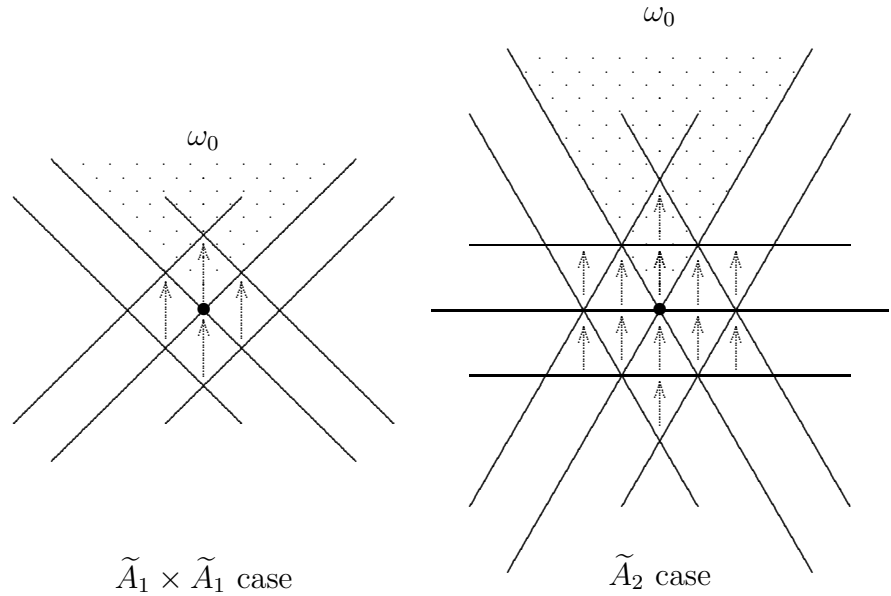
\label{arrows up}
\hfil
\beginpicture
\setcoordinatesystem units <0.5cm,0.866cm>  point at -6 0  % sets scale
\setplotarea x from -5 to 5, y from -4 to 4   % sets plot size up
\put {$\bullet$}      at   0  0
\put {$\w_0$}         at   0  4.5
\put {$\tA_2$ case}   at   0 -4
\putrule from -6    0     to  6    0
\putrule from -5    1     to  5    1
\putrule from -5  -1     to  5  -1
\setlinear
\plot -5  3   1 -3 /
\plot -4  4   4 -4 /
\plot -1  3   5 -3 /
\plot -5 -3   1  3 /
\plot -4 -4   4  4 /
\plot -1 -3   5  3 /
\vshade -4 4 4   <,z,z,z>    0 0 4   <z,z,z,>  4 4 4 /
\setdots <1pt>
\arrow <6pt> [.3,.67] from   0   0.25 to   0   0.9
\arrow <6pt> [.3,.67] from -1 -0.75 to -1 -0.1
\arrow <6pt> [.3,.67] from   1 -0.75 to   1 -0.1
\arrow <6pt> [.3,.67] from   0   0.25 to   0   0.9
\arrow <6pt> [.3,.67] from   0 -1.75 to   0 -1.1
\arrow <6pt> [.3,.67] from -2   0.25 to -2   0.9
\arrow <6pt> [.3,.67] from   2   0.25 to   2   0.9
\arrow <6pt> [.3,.67] from -1   0.1   to  -1   0.75
\arrow <6pt> [.3,.67] from   1   0.1   to    1   0.75
\arrow <6pt> [.3,.67] from   0 -0.9   to    0 -0.25
\arrow <6pt> [.3,.67] from   0   1.1   to    0   1.75
\arrow <6pt> [.3,.67] from -2 -0.9   to  -2 -0.25
\arrow <6pt> [.3,.67] from   2 -0.9   to    2 -0.25
\setcoordinatesystem units <0.5cm,0.5cm> point at 6 0   % sets scale
\setplotarea x from -5 to 5, y from -4 to 4              % sets plot size up
\put {$\bullet$}                 at   0  0
\put {$\w_0$}                    at   0  4.5
\put {$\tA_1\times\tA_1$ case}   at   0 -7
\setsolid
\plot -5  3   1 -3 /
\plot -4  4   4 -4 /
\plot -1  3   5 -3 /
\plot -5 -3   1  3 /
\plot -4 -4   4  4 /
\plot -1 -3   5  3 /
\vshade -4 4 4   <,z,z,z>    0 0 4   <z,z,z,>  4 4 4 /
\setdots <1pt>
\arrow <6pt> [.3,.67] from   0   0.25 to   0   1.75
\arrow <6pt> [.3,.67] from -1 -0.75 to -1  0.75
\arrow <6pt> [.3,.67] from   1 -0.75 to   1  0.75
\arrow <6pt> [.3,.67] from   0 -1.75 to   0 -0.25
\endpicture
\hfil
\caption{Arrows indicating preferred direction.}
\end{figure}
This labelling of the chambers of $\cA_0$ induces a labelling of all
chambers of $\D$ via the retraction~$r$. We say a line segment is {\bf
horizontal} if it is perpendicular to the up direction. In particular,
a horizontal line segment can not be part of a sector wall for any
sector representing~$\w_0$.

It is helpful to imagine that~$\D$ has been folded flat over~$\cA_0$
according to the retraction~$r$ so as to hang down from~$\w_0$. Note that,
as in Figure~\ref{arrows up}, each edge of the apartment~$\cA_0$ is incident
with one chamber which lies above it and one chamber which lies below it.
Given any edge~$(u,v)$ of~$\D$, there is a unique chamber
of~$\D$ containing~$(u,v)$ which retracts to the chamber lying
above~$(r(u), r(v))$ and all other chambers containing~$(u,v)$ retract to the
chamber of~$\cA_0$ lying below~$(r(u),r(v))$.

Suppose that~$C$ and~$C'$ are adjacent chambers of~$\D$. Then
either~$r(C)$ is adjacent to but distinct from~$r(C')$, or~$r(C)$
equals~$r(C')$. In the first case the arrows on the two chambers are
parallel and in the second case they mirror each other through the
common edge and point in converging directions. If~$\cA$ is an
apartment of~$\D$, then the retraction $r : \cA \to \cA_0$ need not be
injective.
The above constraints then lead to overall labellings of
its chambers as in Figure~\ref{generic labelling},
\refstepcounter{picture}
\begin{figure}[htbp]
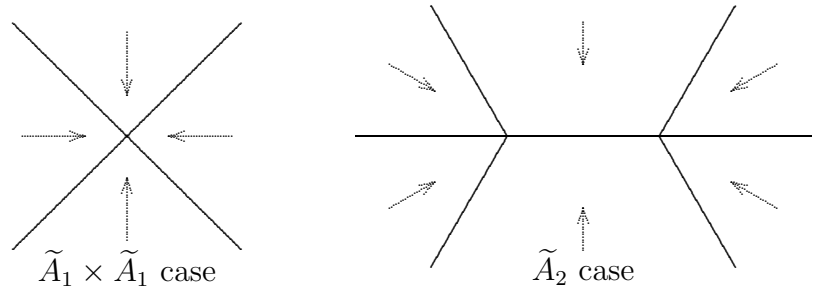
\label{generic labelling}
\hfil
\beginpicture
\setcoordinatesystem units <0.5cm,0.866cm> point at -6 0    % sets scale
\setplotarea x from -4 to 4, y from -2.5 to 2.5         % sets plot size up
\put {$\tA_2$ case}   at   0 -2
\putrule from -6    0     to  6    0
\setlinear
\plot -4  2   -2  0   -4 -2 /
\plot  4 -2   2 0   4 2 /
\setdots <1pt>
\arrow <6pt> [.3,.67] from   0    -1.75 to   0 -1.1
\arrow <6pt> [.3,.67] from   0      1.75 to   0   1.1
\arrow <6pt> [.3,.67] from   5.1   -1.1 to   3.9 -0.7
\arrow <6pt> [.3,.67] from   5.1     1.1 to   3.9   0.7
\arrow <6pt> [.3,.67] from -5.1   -1.1 to -3.9 -0.7
\arrow <6pt> [.3,.67] from -5.1     1.1 to -3.9   0.7
\setcoordinatesystem units <0.5cm,0.5cm> point at 6 0  % sets scale
\setplotarea x from -4 to 4, y from -4 to 3         % sets plot size up
\put {$\tA_1\times\tA_1$ case}   at   0 -3.5
\setsolid
\plot -3  3   0  0   -3 -3 /
\plot  3 -3   0  0    3  3 /
\setdots <1pt>
\arrow <6pt> [.3,.67] from   0    -2.75 to   0  -1.1
\arrow <6pt> [.3,.67] from   0     2.75 to   0   1.1
\arrow <6pt> [.3,.67] from  -2.75  0    to  -1.1 0
\arrow <6pt> [.3,.67] from   2.75  0    to   1.1 0
\endpicture
\hfil
\caption{Generic labelling of apartments.}
\end{figure}
with various degenerate cases also possible. This pattern of arrows is called
the {\bf folding diagram} for~$\cA$, as it indicates how~$\cA$ is folded by~$r$.
The folding diagram is a convenient visual convention for encoding geometric
properties of the retraction. In particular, two apartments have the same
folding diagram if their retractions are equivalent up to translation
in~$\cA_0$.
One may similarly discuss the folding diagram of any connected subset of
an apartment~$\cA$.
Note that in certain degenerate cases it becomes necessary for us to
consider the folding diagram of a line segment in~$\D$ where there
are no chambers for reference. The folding obtained in such a case
will be a piecewise straight line and we trust that the reader will have little
difficulty in appreciating the possibilities.

For clarity of exposition we will now consider the $\tA_1\times\tA_1$
and $\tA_2$ cases separately. Although the ideas used to tackle the
$\tA_1\times\tA_1$ case are also used in the proof of the $\tA_2$
case, the extra complexity of the $\tA_2$ case would make a joint
exposition unwieldy.

\section{The $\tA_1\times\tA_1$ Case}\label{A1 case}

\subsection*{Proof of Theorem~\ref{gpd-thm} in
the~{\boldmath{$\tA_1\times\tA_1$}} case.}

Let $\Wmn=\{ \gamma\in\cG\st \sh(\gamma)=(m,n)\}$ and fix an element
$g\in\CC\cG$ whose support is contained in $\Wmn$, i.e.\ $g\in\CC\Wmn$.
The matrix coefficients of $\rho(g)$ are given by
$\langle \rho(g) \delta_x,\delta_y\rangle =g(x^{-1}y)$
for $x,y\in\vD$ where~$x^{-1}y$ is defined as in~\S\ref{free actions}. Thus,
\[
\langle \rho(g) \delta_x,\delta_y\rangle =
\begin{cases}
g(x^{-1}y) &\text{ if } x^{-1}y\in\Wmn \\
0 & \text{ otherwise.}
\end{cases}
\]
Geometrically, $\langle \rho(g) \delta_x,\delta_y\rangle$ is non-zero
only if $x$ and $y$ are opposite vertices of a rectangle
$\Pmn(x,y)$ of the sort pictured in Figure~\ref{rectangle}.
\refstepcounter{picture}
\begin{figure}[htbp]\label{rectangle}
\hfil
\beginpicture
\setcoordinatesystem units <0.5cm,0.5cm>    % sets scale
\setplotarea x from -6 to 6, y from -2.5 to 4.5         % sets plot size up
\put{$m$} [tl]   at   2.1 -1.1
\put{$m$} [br]  at -2.1   3.1
\put{$n$}  [tr]   at -1.1 -0.1
\put{$n$}  [bl]   at   1.1   2.1
\put{$\bullet$} at    1   -2
\put{$x$}  [t]    at    1    -2.2
\put{$\bullet$} at   -1   4
\put{$y$}  [b]    at -1      4.2
\arrow <10pt> [.3,.67] from   1    -2   to     3   0
\arrow <10pt> [.3,.67] from  -3     2   to   -1   4
\arrow <10pt> [.3,.67] from  -1     4   to     3   0
\arrow <10pt> [.3,.67] from  -3     2   to     1 -2
\endpicture
\hfil
\caption{Convex hull of two vertices in $\D$.}
\end{figure}
This rectangle is the convex hull of~$x$ and~$y$,
and as such is common to all apartments containing both~$x$
and~$y$.

Consider the folding diagram of $\Pmn(x,y)$ induced by the retraction $r$.
In general this will be as indicated in
Figure~\ref{A1 labelling}, although certain degeneracies are possible.
\refstepcounter{picture}
\begin{figure}[htbp]
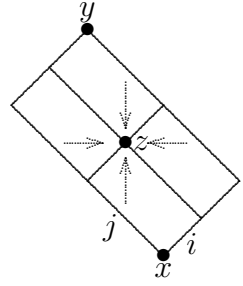
\label{A1 labelling}
\hfil
\beginpicture
\setcoordinatesystem units <0.5cm,0.5cm>    % sets scale
\setplotarea x from -3 to 3, y from -2.5 to 4.5         % sets plot size up
\plot 1  -2    3  0   -1  4     -3 2   1 -2   /
\plot 2 -1   -2 3 /
\plot -1 0   1 2   /
\put{$\bullet$} at  1    -2
\put{$x$} [t]   at  1    -2.2
\put{$\bullet$} at -1     4
\put{$y$} [b]   at -1     4.2
\put{$\bullet$} at  0     1
\put{$z$} [l]   at  0.2   1
\put{$i$} [tl] at  1.6   -1.4
\put{$j$} [tr] at  -0.2   -0.9
\setdots <1pt>
\arrow <6pt> [.3,.67] from    0    -0.6   to    0     0.6
\arrow <6pt> [.3,.67] from    0     2.6   to    0     1.4
\arrow <6pt> [.3,.67] from   -1.6   1     to   -0.6   1
\arrow <6pt> [.3,.67] from    1.6    1     to    0.6   1
\endpicture
\hfil
\caption{General labelling of a rectangle in $\D$.}
\end{figure}
Let $z$ be the apex of the upward labelled subrectangle based at $x$ such that
$x^{-1}z\in W_{i,j}$ with $i$ and $j$ maximal, as in the diagram. We call
$z$ the
{\bf focal point} of the folding.

Associated with each rectangle $\Pmn(x,y)$ is an abstract diagram
$D(x,y)=D_{i,j}$, which is a copy of $\Pmn(x,y)$ in which the
labels of the vertices are forgotten, but the arrows are retained.
Denote by $\cD(m,n)$ the set of all possible diagrams $D_{i,j}$ for
fixed~$m,n\in\NN$.

\begin{lem}
$\#\cD(m,n)=(m+1)(n+1)$.
\end{lem}
\begin{proof}
In order to determine the number of possible diagrams it is sufficient to
enumerate the
possible locations of the focal point of the folding. Hence, there are
$(m+1)(n+1)$ possible diagrams.
\end{proof}

For each $D\in \cD(m,n)$, define an operator $T_D$ on $\ell^2(\vD)$ by
\[
\langle T_D \delta_x,\delta_y\rangle =
\begin{cases} g(x^{-1}y) & \text{ if } D(x,y)=D \\
0 & \text{ otherwise.}\end{cases}
\]
Then $\rho(g)=\displaystyle\sum_{D\in\cD(m,n)}T_D$ and so
\begin{equation}%\label{1}
\| \rho(g)\|\leq \sum_{D\in\cD(m,n)} \| T_D\| .
\end{equation}
We now fix $D=D_{i,j}\in \cD(m,n)$ and proceed to estimate $\| T_D\|$. For each
$z\in\vD$ define
\[
\cH_{z}=\langle \delta_x \st x^{-1}z\in W_{i,j}
\text{ and } \cS_{z}(\w_0)\subseteq \cS_x(\w_0) \rangle
\]
and
\[
\cK_{z}=\langle \delta_y \st z^{-1}y\in W_{m-i,n-j}
\text{ and } \cS_{z}(\w_0)\subseteq \cS_y(\w_0) \rangle .
\]
These give rise to two decompositions;
$\ell^2(\vD)=\oplus_{z\in\vD}\cH_{z}=\oplus_{z\in\vD}\cK_{z}$. Define an
operator $\Tzz  : \cH_{z}\longrightarrow \cK_{z}$ by
\[
\langle \Tzz \delta_x,\delta_y\rangle = g(x^{-1}y)  \text{ for } x\in\cH_{z}
\text{ and } y\in\cK_{z} ,
\]
and define $T_{z',z} : \cH_{z}\longrightarrow \cK_{z'}$ to be zero for
$z'\neq z$. Then $T_D$ can be expressed as a block diagonal operator matrix
$T_D=[T_{z',z}]=\oplus_{z\in\vD}\Tzz $ and it is sufficient to show that for
fixed $z\in\vD$, $\| \Tzz\| \leq  \| g\|_2$.

So fix~$z$ and suppose that $\delta_x\in\cH_z$ and $\delta_y\in\cK_z$.
According to Lemma~\ref{shape decomposition for groupoid}, $x$ and~$y$
are uniquely determined by~$x^{-1}y$.  So each $\gamma\in\Wmn$
contributes at most one matrix coefficient of $\Tzz$ in the
form~$g(\gamma)$.  Therefore we have
$\left\| \Tzz \right\| \leq \left\| \Tzz \right\|_{HS} \leq \| g\|_2$.

This concludes our proof of Theorem~\ref{gpd-thm} in the~$\tA_1\times\tA_1$
case. \hfill $\qedsymbol$

\section{The $\tA_2$ Case}\label{A2 case}

\subsection*{Proof of Theorem~\ref{gpd-thm} in the~{\boldmath{$\tA_2$}} case.}

Let $\Wmn=\{ \gamma\in\cG: \sh(\gamma)=(m,n)\}$ and fix an element
$g\in\CC\cG$ whose support is contained in $\Wmn$, i.e.\ $g\in\CC\Wmn$. The
matrix coefficients of $\rho(g)$ are given by
$\langle \rho(g) \delta_x,\delta_y\rangle =g(x^{-1}y)$
for $x,y\in\vD$. Thus,
\[
\langle \rho(g) \delta_x,\delta_y\rangle =
\begin{cases}
g(x^{-1}y) &\text{ if } x^{-1}y\in\Wmn \\
0 & \text{ otherwise.}
\end{cases}
\]
Geometrically, $\langle \rho(g) \delta_x,\delta_y\rangle$ is non-zero only if
$x$ and $y$ are opposite vertices of a parallelogram $\Pmn(x,y)$ of the sort
pictured in Figure~\ref{parallelogram}.
\refstepcounter{picture}
\begin{figure}[htbp]\label{parallelogram}
\hfil
\beginpicture
\setcoordinatesystem units <0.5cm,0.866cm>    % sets scale
\setplotarea x from -6 to 6, y from -2.5 to 4.5         % sets plot size up
\put{$m$}  [tl]  at  2.1 -1.1
\put{$m$}  [br]  at -2.1  3.1
\put{$n$}  [tr]  at -1.1 -0.1
\put{$n$}  [bl]  at  1.1  2.1
\put{$\bullet$}  at  1   -2
\put{$x$}  [t]   at  1   -2.2
\put{$\bullet$}  at -1    4
\put{$y$}  [b]   at -1    4.2
\arrow <10pt> [.3,.67] from   1    -2   to     3   0
\arrow <10pt> [.3,.67] from  -3     2   to    -1   4
\arrow <10pt> [.3,.67] from  -1     4   to     3   0
\arrow <10pt> [.3,.67] from  -3     2   to     1  -2
\endpicture
\hfil
\caption{Convex hull of two vertices in $\D$.}
\end{figure}
This parallelogram is the convex hull of~$x$ and~$y$, and as such is common to
all apartments containing both~$x$ and~$y$. Note that~$x$ and~$y$ are at the
acute angles of the parallelogram.

Consider the folding of~$\Pmn(x,y)$ induced by the retraction $r$. This will
be one of the possibilities depicted in Figure~\ref{possibilities},
with possible degeneracies.
\refstepcounter{picture}
\begin{figure}[htbp]
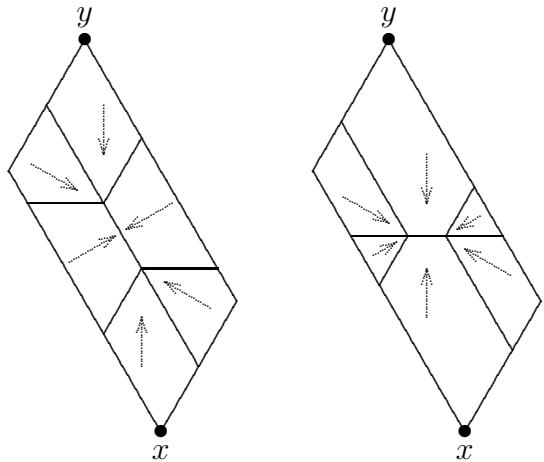
\label{possibilities}
\hfil
\beginpicture
\setcoordinatesystem units <0.5cm,0.866cm>    % sets scale
\setplotarea x from -3 to 11, y from -2.5 to 4.5         % sets plot size up
\putrule from  0.5 0.5 to  2.5 0.5
\putrule from -2.5 1.5 to -0.5 1.5
\plot 1  -2    3  0   -1  4     -3 2   1 -2   /
\plot 2 -1   -2 3 /
\plot -0.5 -0.5   0.5 0.5   /
\plot -0.5  1.5    0.5 2.5   /
\put{$\bullet$} at  1    -2
\put{$x$} [t]   at  1    -2.2
\put{$\bullet$} at -1     4
\put{$y$} [b]   at -1     4.2
\setdots <1pt>
\arrow <6pt> [.3,.67] from    0.5  -1     to    0.5   -0.25
\arrow <6pt> [.3,.67] from   -0.5   3     to   -0.5    2.25
\arrow <6pt> [.3,.67] from    1.3   1.5   to    0.1    1.1
\arrow <6pt> [.3,.67] from   -1.4   0.6   to   -0.2    1
\arrow <6pt> [.3,.67] from    2.3  -0.1   to    1.1    0.3
\arrow <6pt> [.3,.67] from   -2.4   2.1   to   -1.2    1.7
%
%\put{ or } at 4 1
%
\setsolid
\putrule from  6 1 to  10  1
\plot 9  -2    11  0   7  4     5 2   9 -2   /
\plot  7.5   1      5.75 2.75 /
\plot 10.25 -0.75   8.5  1 /
\plot  8.5   1      9.25 1.75  /
\plot  7.5   1      6.75 0.25   /
\put{$\bullet$} at  9    -2
\put{$x$} [t]   at  9    -2.2
\put{$\bullet$} at  7     4
\put{$y$} [b]   at  7     4.2
\setdots <1pt>
\arrow <6pt> [.3,.67] from    8    -0.25  to    8      0.5
\arrow <6pt> [.3,.67] from    8     2.25  to    8      1.5
\arrow <6pt> [.3,.67] from    9.4   1.3   to    8.8    1.1
\arrow <6pt> [.3,.67] from    6.6   0.7   to    7.2    0.9
\arrow <6pt> [.3,.67] from    10.2  0.4   to    9      0.8
\arrow <6pt> [.3,.67] from    5.8   1.6   to    7      1.2
\endpicture
\hfil
\caption{Possible labellings of parallelograms.}
\end{figure}
Let~$z_x$ be the apex of the upward labelled region based at~$x$
such that~$x^{-1}z_x\in W_{i,j}$ with~$i$ and~$i+j$ maximal. Similarly,
let~$z_y$ be the apex of the upward labelled region based at $y$ such that
$y^{-1}z_y\in W_{k,l}$ with~$k$ and~$k+l$ maximal. In terms of our diagrams,
$z_x$ and $z_y$ would be as labelled in Figure~\ref{zs in A2}.

\refstepcounter{picture}
\begin{figure}[htbp]
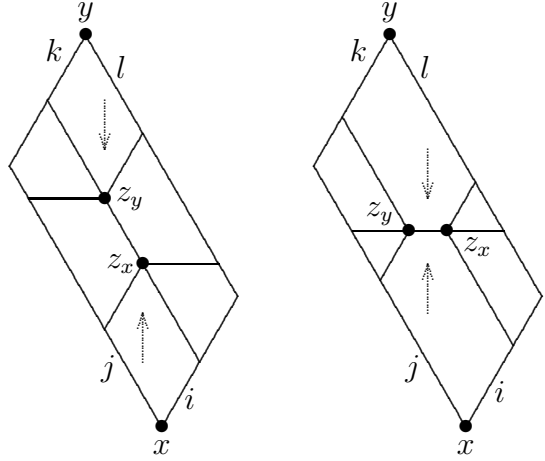
\label{zs in A2}
\hfil
\beginpicture
\setcoordinatesystem units <0.5cm,0.866cm>    % sets scale
\setplotarea x from -3 to 11, y from -2.5 to 4.5         % sets plot size up
\putrule from  0.5 0.5 to  2.5 0.5
\putrule from -2.5 1.5 to -0.5 1.5
\setlinear
\plot 1  -2    3  0   -1  4     -3 2   1 -2   /
\plot 2 -1   -2 3 /
\plot -0.5 -0.5   0.5 0.5   /
\plot -0.5  1.5    0.5 2.5   /
\put{$\bullet$} at  1    -2
\put{$x$} [t]   at  1    -2.2
\put{$\bullet$} at -1     4
\put{$y$} [b]   at -1     4.2
\put{$\bullet$}  at  -0.5   1.5
\put{${z_y}$} [l] at  -0.2   1.5
\put{$\bullet$} at  0.5 0.5
\put{${z_x}$} [r] at  0.3   0.5
\put{$i$} [tl] at  1.6   -1.4
\put{$j$} [tr] at  -0.2   -0.9
\put{$k$} [br] at  -1.6   3.6
\put{$l$} [bl] at  -0.2   3.35
\setdots <1pt>
\arrow <6pt> [.3,.67] from    0.5  -1     to    0.5   -0.25
\arrow <6pt> [.3,.67] from   -0.5   3     to   -0.5    2.25
%
%\put{ or } at 4 1
%
\setsolid
\putrule from  6 1 to  10  1
\plot 9  -2    11  0   7  4     5 2   9 -2   /
\plot  7.5   1      5.75 2.75 /
\plot 10.25 -0.75   8.5  1 /
\plot  8.5   1      9.25 1.75  /
\plot  7.5   1      6.75 0.25   /
\put{$\bullet$}  at  9    -2
\put{$x$} [t]    at  9    -2.2
\put{$\bullet$}  at  7     4
\put{$y$} [b]    at  7     4.2
\put{$\bullet$}  at  7.5   1
\put{${z_y}$} [br] at  7.1   1.1
\put{$\bullet$}  at  8.5   1
\put{${z_x}$} [tl] at  8.92   0.9
\put{$i$} [tl] at  9.75  -1.3
\put{$j$} [tr] at  7.8   -0.9
\put{$k$} [br] at  6.4   3.6
\put{$l$} [bl] at  7.8   3.35
\setdots <1pt>
\arrow <6pt> [.3,.67] from    8    -0.25  to    8      0.5
\arrow <6pt> [.3,.67] from    8     2.25  to    8      1.5
\endpicture
\hfil
\caption{Positions of the vertices $z_x$ and $z_y$.}
\end{figure}

We shall refer to~$z_x$ and~$z_y$ as the {\bf focal points} of the
folding. If $z_x\neq z_y$, the convex hull of~$z_x$ and~$z_y$ is a
line segment which we denote $z_x\twohd z_y$ and the retraction~$r$
maps~$z_x\twohd z_y$ injectively to a horizontal line
segment~$r(z_x\twohd z_y)$ in~$\cA_0$.  Note that the positions of
the focal points completely determine the folding. Among possible
degeneracies, we could have~$z_x=x$ and~$z_y=y$, or~$z_x=z_y$.

Associated to each parallelogram~$\Pmn(x,y)$ is an abstract diagram
$D(x,y)=D_{i,j}^{k,l}$, which is a copy of~$\Pmn(x,y)$ in which the
labels of the vertices are forgotten, but the arrows are retained as in
Figure~\ref{arrows only}.

\refstepcounter{picture}
\begin{figure}[htbp]
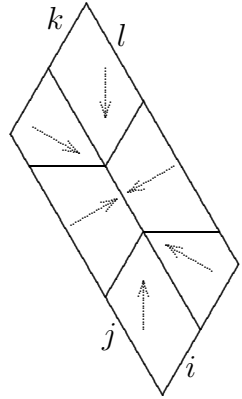
\label{arrows only}
\hfil
\beginpicture
\setcoordinatesystem units <0.5cm,0.866cm>    % sets scale
\setplotarea x from -6 to 6, y from -2 to 4         % sets plot size up
\putrule from  0.5 0.5 to  2.5 0.5
\putrule from -2.5 1.5 to -0.5 1.5
\setlinear
\plot 1  -2    3  0   -1  4     -3 2   1 -2   /
\plot 2 -1   -2 3 /
\plot -0.5 -0.5   0.5 0.5   /
\plot -0.5  1.5    0.5 2.5   /
\put{$i$} [tl] at  1.6   -1.4
\put{$j$} [tr] at  -0.2   -0.9
\put{$k$} [br] at  -1.6   3.6
\put{$l$} [bl] at  -0.2   3.35
\setdots <1pt>
\arrow <6pt> [.3,.67] from    0.5  -1     to    0.5   -0.25
\arrow <6pt> [.3,.67] from   -0.5   3     to   -0.5    2.25
\arrow <6pt> [.3,.67] from    1.3   1.5   to    0.1    1.1
\arrow <6pt> [.3,.67] from   -1.4   0.6   to   -0.2    1
\arrow <6pt> [.3,.67] from    2.3  -0.1   to    1.1    0.3
\arrow <6pt> [.3,.67] from   -2.4   2.1   to   -1.2    1.7
\endpicture
\hfil
\caption{The diagram $D_{i,j}^{k,l}$.}
\end{figure}

Denote by $\cD(m,n)$ the set of all possible diagrams $D_{i,j}^{k,l}$ for fixed
$m,n\in\NN$. We divide the family $\cD(m,n)$ into two classes as follows. Let
$\cD'(m,n)$ be the set of diagrams associated to parallelograms $P_{m,n}(x,y)$
in which $x^{-1}z_x\in W_{p,q}$ and $x^{-1}z_y\in W_{s,t}$ where $p+q=s+t$.
Thus a diagram $D\in\cD(m,n)$ is in $\cD'(m,n)$ if it is of the form
depicted in Figure~\ref{E(m,n)} or if the focal points are coincident.
\refstepcounter{picture}
\begin{figure}[htbp]\label{E(m,n)}
\hfil
\beginpicture
\setcoordinatesystem units <0.5cm,0.866cm>    % sets scale
\setplotarea x from 5 to 11, y from -2 to 4         % sets plot size up
\setsolid
\putrule from  6 1 to  10  1
\plot 9  -2    11  0   7  4     5 2   9 -2   /
\plot  7.5   1      5.75 2.75 /
\plot 10.25 -0.75   8.5  1 /
\plot  8.5   1      9.25 1.75  /
\plot  7.5   1      6.75 0.25   /
\setdots <1pt>
\arrow <6pt> [.3,.67] from    8    -0.25  to    8      0.5
\arrow <6pt> [.3,.67] from    8     2.25  to    8      1.5
\arrow <6pt> [.3,.67] from    9.4   1.3   to    8.8    1.1
\arrow <6pt> [.3,.67] from    6.6   0.7   to    7.2    0.9
\arrow <6pt> [.3,.67] from    10.2  0.4   to    9      0.8
\arrow <6pt> [.3,.67] from    5.8   1.6   to    7      1.2
\endpicture
\hfil
\caption{Parallelograms in $\cD'(m,n)$.}
\end{figure}
Let $\cD''(m,n)=\cD(m,n)\setminus\cD'(m,n)$. So $D\in\cD(m,n)$ is in
$\cD''(m,n)$ if it is of one of the forms depicted in
Figure~\ref{F(m,n)} with distinct focal points.
\refstepcounter{picture}
\begin{figure}[htbp]\label{F(m,n)}
\hfil
\beginpicture
\setcoordinatesystem units <0.5cm,0.866cm> point at 4 1   % sets scale
\setplotarea x from -3 to 3, y from -2 to 4         % sets plot size up
\putrule from  0.5 0.5 to  2.5 0.5
\putrule from -2.5 1.5 to -0.5 1.5
\plot 1  -2    3  0   -1  4     -3 2   1 -2   /
\plot 2 -1   -2 3 /
\plot -0.5 -0.5   0.5 0.5   /
\plot -0.5  1.5    0.5 2.5   /
\setdots <1pt>
\arrow <6pt> [.3,.67] from    0.5  -1     to    0.5   -0.25
\arrow <6pt> [.3,.67] from   -0.5   3     to   -0.5    2.25
\arrow <6pt> [.3,.67] from    1.3   1.5   to    0.1    1.1
\arrow <6pt> [.3,.67] from   -1.4   0.6   to   -0.2    1
\arrow <6pt> [.3,.67] from    2.3  -0.1   to    1.1    0.3
\arrow <6pt> [.3,.67] from   -2.4   2.1   to   -1.2    1.7
%
%\put{ or } at 4 1
%
\setcoordinatesystem units <0.5cm,0.866cm> point at -4 1   % sets scale
\setplotarea x from -3 to 3, y from -2 to 4         % sets plot size up
\setsolid
\putrule from  -0.5 0.5 to  -2.5 0.5
\putrule from 2.5 1.5 to 0.5 1.5
\plot -1  -2    -3  0   1  4     3 2   -1 -2   /
\plot -2 -1   2 3 /
\plot 0.5 -0.5   -0.5 0.5   /
\plot 0.5  1.5   -0.5 2.5   /
\setdots <1pt>
\arrow <6pt> [.3,.67] from   -0.5  -1     to   -0.5   -0.25
\arrow <6pt> [.3,.67] from    0.5   3     to    0.5    2.25
\arrow <6pt> [.3,.67] from   -1.3   1.5   to   -0.1    1.1
\arrow <6pt> [.3,.67] from    1.4   0.6   to    0.2    1
\arrow <6pt> [.3,.67] from   -2.3  -0.1   to   -1.1    0.3
\arrow <6pt> [.3,.67] from    2.4   2.1   to    1.2    1.7
\endpicture
\hfil
\caption{Parallelograms in $\cD''(m,n)$.}
\end{figure}

\begin{lem}\label{parallel estimate}
$\#\cD''(m,n) = 1/2(m+1)(n+1)(m+n)$.
\end{lem}
\begin{proof}
In order to determine the number of diagrams in $\cD''(m,n)$ it is sufficient to
enumerate the possible positions of the focal points in the foldings. There are
$(m+1)(n+1)$ choices for the first focal point. Given one focal point, the
possible positions of the other are indicated in Figure~\ref{z2 positions}.
So there are $m+n$ positions the second focal point could occupy.
Since the order of choice is irrelevant we must divide the total number of
possibilities by a factor of $2$. Hence the result.
\refstepcounter{picture}
\begin{figure}[htbp]
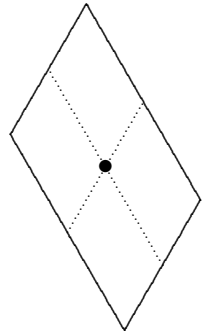
\label{z2 positions}
\hfil
\beginpicture
\setcoordinatesystem units <0.5cm,0.866cm>    % sets scale
\setplotarea x from -6 to 6, y from -2 to 3         % sets plot size up
\put{$\bullet$} at 0.5 0.5
\plot 1  -2    3  0   0  3     -2 1   1 -2   /
\setdots <2pt>
\plot -0.5 -0.5   1.5 1.5   /
\plot 2 -1   -1 2 /
\endpicture
\hfil
\caption{Possible positions of second focal point.}
\end{figure}
\end{proof}

For each $D\in \cD(m,n)$, define an operator $T_D$ on
$\ell^2(\vD)$ by
\[
\langle T_D \delta_x,\delta_y\rangle =
\begin{cases} g(x^{-1}y) & \text{ if } D(x,y)=D \\
0 & \text{ otherwise.}\end{cases}
\]
Thus
$\rho(g)=\displaystyle\sum_{D\in\cD(m,n)}T_D=
\displaystyle\sum_{D\in\cD'(m,n)}T_D+\displaystyle\sum_{D\in\cD''(m,n)}T_D$.

For each diagram $D\in\cD'(m,n)$ choose one of the focal points and call it a
{\bf distinguished focal point}. There are $(m+1)(n+1)$ possible positions for
the distinguished focal point. Let $\cD'_{i,j}(m,n)$ be the set of diagrams in
$\cD'(m,n)$ whose distinguished focal point is in position $(i,j)$, so that
\[
\displaystyle\sum_{D\in\cD'(m,n)}T_D=
\displaystyle\sum_{i=0}^m\sum_{j=0}^n\sum_{D\in\cD'_{i,j}(m,n)}T_D.
\]
An argument analogous to that used in~\S\ref{A1 case} shows
that
\[
\left\|\displaystyle\sum_{D\in\cD'_{i,j}(m,n)}T_D\right\|\leq \| g\|_2.
\]
Hence
\[
\left\|\displaystyle\sum_{D\in\cD'(m,n)}T_D\right\|\leq (m+1)(n+1)\| g\|_2
\]
and so
\begin{equation}\label{1}
\| \rho(g)\|\leq (m+1)(n+1)\|g\|_2 +\sum_{D\in\cD''(m,n)} \| T_D\| .
\end{equation}
We proceed to show that, for $D\in\cD''(m,n)$, $\left\|T_D\right\|\leq
\sqrt{\max(m,n)+1} \|g\|_2$. Since
\[
(m+1)(n+1)+\frac{1}{2}(m+1)(n+1)(m+n)=\frac{1}{2}(m+1)(n+1)(m+n+2)
\]
we will then have established that
\[
\|\rho(g)\|\leq
1/2(m+1)(n+1)(m+n+2)\sqrt{\max(m,n)+1}\;\|g\|_2.
\]

We now fix $D=D_{i,j}^{k,l}\in \cD''(m,n)$ and estimate
$\| T_D\|$. Note that, if $D(x,y)=D$, then $z_x^{-1}z_y\in W_{s,t}$ where
$(s,t)= (p,0)$ or $(0,q)$, and $(s,t)$ is determined by~$D$. Thus
$i,j,k,l,s,t$ are all now fixed.

Given $z\in\vD$, define
\[
\cH_{z}=\langle \delta_x \st x^{-1}z\in W_{i,j}
\text{ and } \cS_{z}(\w_0)\subseteq \cS_x(\w_0) \rangle
\]
and
\[
\cK_{z}=\langle \delta_y \st y^{-1}z\in W_{k,l}
\text{ and } \cS_{z}(\w_0)\subseteq \cS_y(\w_0) \rangle .
\]
These give rise to two decompositions;
$\ell^2(\vD)=\oplus_{z\in\vD}\cH_{z}=\oplus_{z\in\vD}\cK_{z}$.

If $(z_1,z_2)\in\vD\times\vD$ with $z_1^{-1}z_2\in W_{s,t}$ and
$r$ maps $z_1\twohd z_2$ injectively to a horizontal line, define
$T_{z_2,z_1} : \cH_{z_1}\longrightarrow \cK_{z_2}$ by
\[
\langle T_{z_2,z_1}\delta_x,\delta_y\rangle = g(x^{-1}y)
\text{ for } \delta_x\in\cH_{z_1}
\text{ and } \delta_y\in\cK_{z_2} ,
\]
and define $T_{z_2,z_1} : \cH_{z_1}\longrightarrow \cK_{z_2}$ to be zero for
other pairs $(z_1,z_2)\in\vD\times\vD$. Then $T_D$ can be expressed as an
operator matrix $T_D=[T_{z_2,z_1}]$ and therefore
\begin{equation}\label{2}
\| T_D\| \leq  \left\| \: \left[ \, \left\| T_{z_2,z_1}
\right\| \, \right] \: \right\|
\end{equation}
by Lemma~\ref{B}. We proceed to estimate $\|T_{z_2,z_1}\|$.

Fix $(z_1,z_2)\in\vD\times\vD$ with $z_1^{-1}z_2\in W_{s,t}$ and
such that $r$ maps $z_1\twohd z_2$ injectively to a horizontal line.
Thus $r(z_1\twohd z_2)$ is a horizontal line of shape~$(s,t)$.
The vectors $\delta_x\in\cH_{z_1}$, $\delta_y\in\cK_{z_2}$ are uniquely
determined by $x^{-1}y\in\cG$ 
according to Lemma~\ref{shape decomposition for groupoid}.
Thus each $\gamma\in\Wmn$ contributes at most
one matrix coefficient of $T_{z_2,z_1}$ in the form $g(\gamma)$.
Define $\tg\in\CC W_{s,t}$ by
\[
\tg(\zeta)= \begin{cases}
\left( \displaystyle
\sum_{\substack{\alpha\in W_{i,j}, \beta\in W_{k,l}\\ \alpha\zeta\beta\in\Wmn}}
|g(\alpha\zeta\beta)|^2 \right)^{1/2}  & \text{ if } \zeta\in W_{s,t} \\
0 & \text{ otherwise. }
\end{cases}
\]
Note that $\| \tg\|_2=\| g\|_2$ and
\begin{equation}\label{3}
\| T_{z_2,z_1}\| \leq \| T_{z_2,z_1}\|_{HS} \leq \tg(z_1^{-1}z_2).
\end{equation}

Define a new operator $\tTD$ on $\ell^2(\vD)$ by
\[
\langle \tTD \delta_{z_1},\delta_{z_2}\rangle =
\begin{cases}
\tg(z_1^{-1}z_2) & \text{ if $z_1^{-1}z_2\in W_{s,t}$ and
$r(z_1\twohd z_2)$ is a horizontal line of shape $(s,t)$,} \\
0 & \text{ otherwise.}
\end{cases}
\]
By~(\ref{2}),~(\ref{3}) and Lemma~\ref{C} we have $\| T_D\|\leq\|\tTD\|$.
Since we also have
$\| \tg\|_2=\| g\|_2$, it will be sufficient to prove
\begin{equation}\label{4}
\|\tTD\|\leq \sqrt{\max(m,n)+1}\;\|\tg\|_2
   .
\end{equation}

Recall that $(s,t)=(p,0)$ with $0\leq p\leq m$ or $(0,q)$ with
$0\leq p\leq n$. We suppose the former and prove that
\begin{equation}\label{5}
\|\tTD\|\leq\sqrt{p+1}\;\|\tg\|_2.
\end{equation}
A similar argument in the case $(s,t)=(0,q)$ gives
$\|\tTD\|\leq\sqrt{q+1}\;\|\tg\|_2$ thus establishing~(\ref{4}).

In order to prove~(\ref{5}), we simplify our notation.
Let $g\in\CC \Wp$ where $0\leq p\leq m$ and define an operator $T$ on
$\ell^2(\vD)$ by
\[
\langle T \delta_x,\delta_y\rangle =
\begin{cases}
g(x^{-1}y) & \text{ if } x^{-1}y\in \Wp \text{ and }
r(x\twohd y)  \text{ is a horizontal line of shape }(p,0),\\
0 & \text{ otherwise.}
\end{cases}
\]
We must prove
\begin{equation}\label{6}
\| T\|\leq \sqrt{p+1}\;\| g\|_2.
\end{equation}

Given $(x,y)\in\vD\times\vD$ with $x^{-1}y\in \Wp$ and such that
$r(x\twohd y)$ is a horizontal line of shape~$(p,0)$, there
is a unique $z=z(x,y)\in\vD$ such that the convex hull
of $x,y$ and $z$ is labelled as in Figure~\ref{unique z} and such that
$z^{-1}x, x^{-1}y, y^{-1}z\in \Wp$. That is $xyz$ is an equilateral triangle,
with side of length $p$, base $xy$ and pointing to $\omega_0$.
\refstepcounter{picture}
\begin{figure}[htbp]
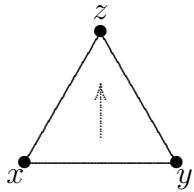
\label{unique z}
\hfil
\beginpicture
\setcoordinatesystem units <0.5cm,0.866cm>    % sets scale
\setplotarea x from -3 to 3, y from -0.5 to 2.5         % sets plot size up
\setlinear
\plot -2 0  2 0  0 2  -2 0 /
\put{$\bullet$} at -2   0
\put{$x$} [tr]  at -2  -0.1
\put{$\bullet$} at  2   0
\put{$y$} [tl]  at  2  -0.1
\put{$\bullet$} at  0   2
\put{$z$} [b]   at  0   2.2
\setdots <1pt>
\arrow <6pt> [.3,.67] from    0   0.4   to   0   1.2
\endpicture
\hfil
\caption{Labelling of convex hull of $x$, $y$ and $z=z(x,y)$.}
\end{figure}
To see this, suppose that $x_1$ is the vertex adjacent to, but
distinct from, $x$ in $x\twohd y$. Let $w$ be the third vertex of the
unique chamber containing $(x,x_1)$ and retracting above
$(r(x),r(x_1))$. The convex hull of $x,y$ and $w$ is a
trapezoidal strip in~$\D$ which has $x\twohd y$ as one of its bases,
contains the chamber with vertices $(x,x_1,w)$ and whose other base,
$w\twohd w'$, has length~$p-1$ as in Figure~\ref{trapezoidal strip}.
\refstepcounter{picture}
\begin{figure}[htbp]\label{trapezoidal strip}
\hfil
\beginpicture
\setcoordinatesystem units <0.5cm,0.866cm>    % sets scale
\setplotarea x from -5 to 5, y from -0.5 to 1.5         % sets plot size up
\setlinear
\plot -4 0  4 0  3 1  -3 1  -4 0 /
\plot -3 1 -2 0 /
\put{$\bullet$} at -4   0
\put{$x$} [tr]  at -4  -0.1
\put{$\bullet$} at -2   0
\put{$x_1$}[t]  at -2  -0.1
\put{$\bullet$} at  4   0
\put{$y$} [tl]  at  4  -0.1
\put{$\bullet$} at -3   1
\put{$w$} [b]   at -3   1.2
\put{$\bullet$} at  3   1
\put{$w'$}[b]   at  3   1.2
\setdots <1pt>
\arrow <6pt> [.3,.67] from    0   0.2   to   0   0.8
\arrow <6pt> [.3,.67] from   -3   0.2   to  -3   0.8
\endpicture
\hfil
\caption{Trapezoidal strip retracting above $x\twohd y$.}
\end{figure}
Since the retraction is contractive,~$r$
must map~$w\twohd w'$ injectively to a horizontal line lying
above~$r(x)\twohd r(y)$. By induction, this argument demonstrates the
existence of the triangle depicted in Figure~\ref{unique z}.
Note that $\cS_z(\w_0)=\cS_x(\w_0)\cap \cS_y(\w_0)$.
For later convenience we define a set
\[
\begin{array}{rl}
\cC=\left\{ \vphantom{r\left( [x,y]\right) }
(x,y,z)\in\vD\times\vD\times\vD
\st
x^{-1}y\in \Wp, \right. & r\left( x\twohd y\right)
\text{ is a horizontal line of shape $(p,0)$,} \\
&\left. \text{ and } \cS_z(\w_0)=\cS_x(\w_0)\cap \cS_y(\w_0)
\vphantom{r\left( [x,y]\right) }\right\} .
\end{array}
\]

For each $z\in\vD$, define
\[
\cH_z=\langle\delta_x \st z^{-1}x\in \Wp\text{ and }
\cS_z(\w_0)\subseteq \cS_x(\w_0)\rangle
\]
and
\[
\cK_z=\langle\delta_y \st y^{-1}z\in \Wp\text{ and }
\cS_z(\w_0)\subseteq \cS_y(\w_0)\rangle .
\]
Thus $\delta_x\in\cH_z$ if $z=z(x,y)$ for some $y\in\vD$ and a similar condition
characterizes the elements $\delta_y\in\cK_z$.
Let $T_{z',z}:\cH_z\longrightarrow\cK_{z'}$ be defined by
\[
\langle T_{z',z} \delta_x,\delta_y\rangle =
g(x^{-1}y) \text{ if } \delta_x\in\cH_z \text{ and }
\delta_y\in\cK_{z'}.
\]
Note that $\ell^2(\vD) =\oplus_{z\in\vD} \cH_z =\oplus_{z'\in\vD} \cK_{z'}$.
Since $T_{z',z}=0$ unless $z'=z$, $T$ can be expressed as a block diagonal
matrix $T=\left[ T_{z',z} \right] = \oplus_{z\in\vD} \Tzz$. Thus it is
sufficient to show that, for fixed $z\in\vD$,
\begin{equation}\label{7}
\| \Tzz \|\leq \sqrt{p+1}\;\| g\|_2.
\end{equation}
We now fix~$z$ and show that
\[
\left| \langle \Tzz f_1, f_2 \rangle\right|
\leq \sqrt{p+1}\;\| f_1 \|_2 \, \|f_2 \|_2 \, \| g\|_2
\]
for $f_1\in\cH_z$, and $f_2\in\cK_z$.  We may assume $f_1$, $f_2$,
$g\geq 0$.  Then
\[
\langle \Tzz f_1, f_2 \rangle =
\sum_{\substack{x,y\in \vD\\ (x,y,z)\in \cC}}
f_1(x)f_2(y)g(x^{-1}y).
\]
We define $g_1,g_2\in\CC\cG$ via
\[
g_1(\alpha)= \begin{cases} f_1(x) & \text{ if } \alpha=z^{-1}x \\
0 & \text{ otherwise}\end{cases}
\quad\text{ and }\quad
g_2(\beta)= \begin{cases} f_2(y) & \text{ if } \beta=y^{-1}z \\
0 & \text{ otherwise}\end{cases}
\]
and change our emphasis to obtain
\begin{equation}\label{8}
\langle \Tzz f_1, f_2 \rangle =
\sum_{\substack{x,y\in \vD\\ (x,y,z)\in \cC}}
g_1(z^{-1}x)g_2(y^{-1}z)g(x^{-1}y).
\end{equation}
We define a set
\[
\cT_p=\left\{ (\alpha,\beta,\gamma)\in\cG\times\cG\times\cG\st
\alpha,\beta,\gamma\in \Wp \text{ and } \gamma\beta\alpha
\text{ is a unit}\right\}
\]
It can be shown that if $x,y,z$ are vertices in an~$\tA_2$ building and
$\sh(x,y)=\sh(y,z)=\sh(z,x)=(p,0)$, the convex hull of $\{ x,y,z\}$
is necessarily an equilateral triangle of side length~$p$ in some apartment.
Hence, given any~$(\alpha,\beta,\gamma)\in\cT_p$ and any~$v\in\vD$ such that
$v\gamma$ exists, the vertices~$v,v\gamma,v\alpha^{-1}$ lie in a common
apartment and their convex hull is an equilateral triangle of side length~$p$.
Use~$\cT_p$ to rewrite equation~(\ref{8}) as
\[
\langle \Tzz f_1, f_2 \rangle =
\sum_{(\alpha,\beta,\gamma)\in \cT_p} \tf_1(\alpha)\tf_2(\beta)\tf_3(\gamma) ,
\]
where $\tf_i\in\CC\Wp$ for $i=1,2,3$. This is equivalent to having changed the
emphasis

{}\hfil
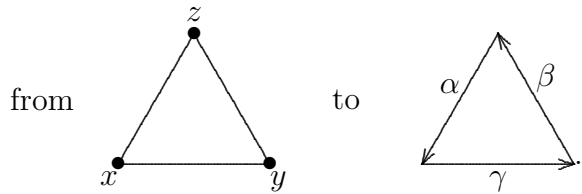
\beginpicture
\setcoordinatesystem units <0.5cm,0.866cm>    % sets scale
\setplotarea x from -3 to 11, y from -0.7 to 2.5         % sets plot size up
\put{from}       at  -3 1
\setlinear
\plot -1 0  3 0  1 2  -1 0 /
\put{$\bullet$}  at  -1   0
\put{$x$} [tr]   at  -1  -0.1
\put{$\bullet$}  at   3   0
\put{$y$} [t l]  at   3  -0.1
\put{$\bullet$}  at   1   2
\put{$z$} [b]    at   1   2.2
\put{to}         at   5   1
\arrow <6pt> [.3,.67] from  7 0 to 11 0
\arrow <6pt> [.3,.67] from 11 0 to  9 2
\arrow <6pt> [.3,.67] from  9 2 to  7 0
\put{$\gamma$} [t]     at   9  -0.1
\put{$\beta$} [bl]    at  10   1.1
\put{$\alpha$} [br]    at   8   1.1
\endpicture.
\hfil{}

We now complete the proof of~(\ref{7}) by proving the following result.

\begin{lem}
If $f_1,f_2,f_3\in\CC\Wp$ then
\[
\left| \sum_{(\alpha,\beta,\gamma)\in \cT_p} f_1(\alpha)f_2(\beta)f_3(\gamma)
\right| \leq
\sqrt{p+1}\;\| f_1\|_2 \, \| f_2\|_2 \, \| f_3\|_2.
\]
\end{lem}
\begin{proof}
Assume that $f_1,f_2,f_3\geq 0$. For a fixed $f_3$, define
$T_3:\CC\Wp \longrightarrow\CC\Wp$ by
\[
\langle T_3\delta_\alpha,\delta_\beta\rangle =
\displaystyle\sum_{\gamma :(\alpha,\beta,\gamma)\in\cT_p} f_3(\gamma)
\]
where the right hand side has at most one non-zero term, and we adopt the
convention that it is zero if the set
$\left\{\gamma\in\Wp\,|\,(\alpha,\beta,\gamma)\in\cT_p\right\}$
is empty. We will use this convention again later. Thus
\[
\sum_{(\alpha,\beta,\gamma)\in \cT_p} f_1(\alpha)f_2(\beta)f_3(\gamma) =
\langle T_3 f_1, f_2\rangle.
\]
We must show that $\| T_3\|\leq\sqrt{p+1}\;\| f_3\|_2$, or equivalently that
$\| T_3^*T_3\|\leq(p+1)\;\| f_3\|_2^2$. Now,
\begin{eqnarray*}
\langle T_3^*T_3\delta_{\alpha_1},\delta_{\alpha_2}\rangle &=&
\langle T_3\delta_{\alpha_1},T_3\delta_{\alpha_2}\rangle \\
&=& \sum_{\beta\in\cG}
\left(\displaystyle
\sum_{\gamma_1 :(\alpha_1,\beta,\gamma_1)\in\cT_p} f_3(\gamma_1)\right)
\left(\displaystyle
\sum_{\gamma_2 :(\alpha_2,\beta,\gamma_2)\in\cT_p} f_3(\gamma_2)\right) \\
&=& \sum_{\substack{\beta,\gamma_1,\gamma_2\in\cG\\
(\alpha_1,\beta,\gamma_1),(\alpha_2,\beta,\gamma_2)\in\cT_p}}
f_3(\gamma_1)f_3(\gamma_2) .
\end{eqnarray*}

We say that $\alpha_1,\alpha_2\in\cG$
{\bf share precisely~{\boldmath{$j$}} initial letters} if
\begin{itemize}
\item there exists a~$z\in\vD$ such that $z\alpha_1$ and $z\alpha_2$ are
both defined, and
\item there exist $\zeta\in W_{j,0}$, and $\talpha_1,\talpha_2\in W_{p-j,0}$
such that $\alpha_1=\zeta\talpha_1$, $\alpha_2=\zeta\talpha_2$, and
\[
\sh(\talpha_1^{-1}\talpha_2)=(p-j,p-j).
\]
\end{itemize}
Note that if $\alpha_1$, $\alpha_2\in W_{p,0}$, and if
$\alpha_1^{-1}\alpha_2$ is defined, then $\alpha_1$ and $\alpha_2$
must share precisely $j$ letters for some $0\leq j\leq p$. We have a
decomposition $T_3^*T_3 =S_0+\cdots +S_p$ where
\[
\langle S_j\delta_{\alpha_1},\delta_{\alpha_2}\rangle =
\begin{cases}
\langle T_3^*T_3\delta_{\alpha_1},\delta_{\alpha_2}\rangle  &
\text{ if } \alpha_1, \alpha_2
\text{ share precisely $j$ initial letters,} \\
0 & \text{ otherwise.}
\end{cases}
\]
It is therefore sufficient to prove that $\| S_j\|\leq\| f_3\|_2^2$
for $0\leq j\leq p$.  Diagrammatically speaking, $\langle
S_j\delta_{\alpha_1},\delta_{\alpha_2}\rangle \neq 0$ only if there
are diagrams

{}\hfil
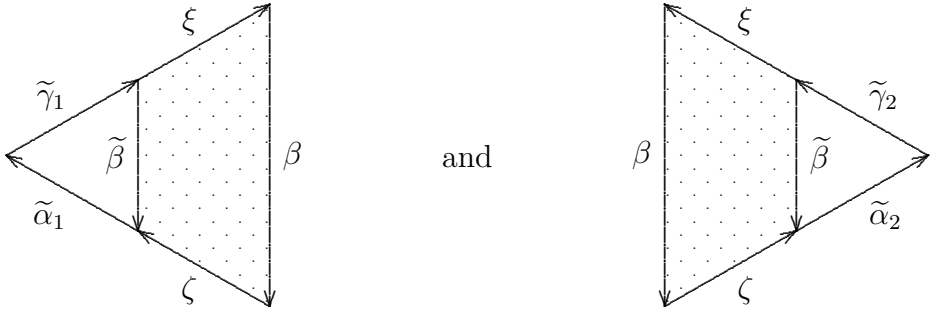
\beginpicture
\setcoordinatesystem units <0.866cm,1cm>    % sets scale
\setplotarea x from -8 to 8, y from -2.5 to 2.5         % sets plot size up
\setlinear
\put{$\beta$}     [l]     at  -2.8  0
\put{$\zeta$}     [tr]    at  -4.1 -1.6
\put{$\talpha_1$} [tr]    at  -6.1 -0.6
\put{$\tgamma_1$} [br]    at  -6.1  0.6
\put{$\xi$}     [br]    at  -4.1  1.6
\put{$\tbeta$}   [r]     at  -5.2  0
\arrow <6pt> [.3,.67] from  -3  2 to  -3 -2
\arrow <6pt> [.3,.67] from  -3 -2 to  -5 -1
\arrow <6pt> [.3,.67] from  -5 -1 to  -7  0
\arrow <6pt> [.3,.67] from  -7  0 to  -5  1
\arrow <6pt> [.3,.67] from  -5  1 to  -3  2
\arrow <6pt> [.3,.67] from  -5  1 to  -5 -1
\vshade  -5 -1 1 -3 -2 2 /
\put{ and }           at  0    0
\put{$\beta$}     [r]     at  2.8  0
\put{$\zeta$}     [tl]    at  4.1 -1.6
\put{$\talpha_2$} [tl]    at  6.1 -0.6
\put{$\tgamma_2$} [bl]    at  6.1  0.6
\put{$\xi$}     [bl]    at  4.1  1.6
\put{$\tbeta$}   [l]     at  5.2  0
\arrow <6pt> [.3,.67] from  3  2 to  3 -2
\arrow <6pt> [.3,.67] from  3 -2 to  5 -1
\arrow <6pt> [.3,.67] from  5 -1 to  7  0
\arrow <6pt> [.3,.67] from  7  0 to  5  1
\arrow <6pt> [.3,.67] from  5  1 to  3  2
\arrow <6pt> [.3,.67] from  5  1 to  5 -1
\vshade 3 -2 2  5 -1 1 /
\endpicture
\hfil{}

\noindent whose common (shaded) sections are identical.

For a fixed $\zeta\in W_{j,0}$, let
\[
\cH_{\zeta}=\langle \delta_\alpha \st \alpha\in\Wp, \alpha=\zeta\talpha,
\talpha\in W_{p-j,0}\rangle.
\]
If $\delta_\alpha\in\cH_\zeta$ then
$\langle S_j\delta_\alpha,\delta_{\alpha'}\rangle=0$ unless
$\delta_{\alpha'}\in\cH_\zeta$. Therefore $S_j\cH_\zeta\subseteq\cH_\zeta$
and it follows that
$S_j={\displaystyle\oplus_{\zeta\in W_{j,0}}} S_j^\zeta$, where
$S_j^\zeta :\cH_\zeta\longrightarrow\cH_\zeta$ is the restriction of
$S_j$ to $\cH_\zeta$.
So it is enough to bound $S_j^\zeta$ for each $\zeta\in W_{j,0}$.
However, in this case
\begin{eqnarray*}
\| S_j^\zeta \|^2 & \leq & \| S_j^\zeta \|^2_{HS}  =
\displaystyle\sum_{\substack{
\alpha_1,\alpha_2\in\G \\
\alpha_1=\zeta\talpha_1  \\
\alpha_2=\zeta\talpha_2 \\
\talpha_1^{-1}\talpha_2\in W_{p-j,p-j}
} }
\left( \displaystyle\sum_{\substack{
\beta,\gamma_1,\gamma_2 \in\Wp \\
(\alpha_1,\beta,\gamma_1)\in\cT_p \\
(\alpha_2,\beta,\gamma_2)\in\cT_p
} }
f_3(\gamma_1)f_3(\gamma_2)  \right)^2 \\
& = & \displaystyle\sum_{\substack{
\talpha_1,\talpha_2\in W_{p-j,0} \\
\talpha_1^{-1}\talpha_2\in W_{p-j,p-j}
} }
\left( \displaystyle\sum_{\substack{
\xi\in W_{j,0},\tbeta,\tgamma_1,\tgamma_2 \in W_{p-j,0} \\
\tgamma_1 \xi, \tgamma_2 \xi, \xi^{-1} \tbeta \zeta^{-1}\in\Wp  \\
(\talpha_1,\tbeta,\tgamma_1),(\talpha_2,\tbeta,\tgamma_2)\in\cT_{p-j}
} }
f_3(\tgamma_1 \xi)f_3(\tgamma_2 \xi)  \right)^2 ,
\end{eqnarray*}
where the last term is obtained from the previous one by using the smaller
triangles formed by $\talpha_i$, $\tbeta$, and $\tgamma_i$ to enumerate the
larger ones formed by $\alpha_i$, $\beta$, and $\gamma_i$. Thus
\begin{eqnarray}\label{triple sum}
\| S_j^\zeta \|^2 \leq \displaystyle\sum_{\substack{
\talpha_1,\talpha_2\in W_{p-j,0} \\
\talpha_1^{-1}\talpha_2\in W_{p-j,p-j}
} }
\left(
\displaystyle\sum_{\substack{
\tbeta,\tgamma_1,\tgamma_2 \in W_{p-j,0} \\
(\talpha_1,\tbeta,\tgamma_1)\in\cT_{p-j} \\
(\talpha_2,\tbeta,\tgamma_2)\in\cT_{p-j}
}}
\left( \displaystyle\sum_{\substack{
\xi\in W_{j,0}\\
\tgamma_1 \xi, \tgamma_2 \xi\in\Wp
}}
f_3(\tgamma_1 \xi)f_3(\tgamma_2 \xi)  \right) \right)^2
\end{eqnarray}
where the middle sum is over all possible diagrams of the form
depicted in Figure~\ref{middle sum}. Since $\talpha_1^{-1}\talpha_2\in
W_{p-j,p-j}$, an application of Lemma~\ref{shape decomposition for
groupoid} shows that this rhombus is entirely determined by
$\talpha_1$ and $\talpha_2$ and so that the middle sum contains only
one term.
\refstepcounter{picture}
\begin{figure}[htbp]
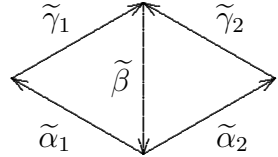
\label{middle sum}
\hfil
\beginpicture
\setcoordinatesystem units <0.866cm,1cm>    % sets scale
\setplotarea x from -2 to 2, y from -1 to 1         % sets plot size up
\setlinear
\put{$\talpha_1$} [tr]    at  -1.1 -0.6
\put{$\tgamma_1$} [br]    at  -1.1  0.6
\put{$\tbeta$}   [r]     at  -0.2  0
\put{$\talpha_2$} [tl]    at   1.1 -0.6
\put{$\tgamma_2$} [bl]    at   1.1  0.6
%\put{.}               at  3    0
\arrow <6pt> [.3,.67] from   0 -1 to  -2  0
\arrow <6pt> [.3,.67] from  -2  0 to   0  1
\arrow <6pt> [.3,.67] from   0  1 to   0 -1
\arrow <6pt> [.3,.67] from   0 -1 to   2  0
\arrow <6pt> [.3,.67] from   2  0 to   0  1
\endpicture
\hfil
\caption{Diagram indexing the middle sum in~(\ref{triple sum}).}
\end{figure}
The rhombus in Figure~\ref{middle sum} is also uniquely determined by
$\tgamma_1$ and $\tgamma_2$ whenever $\tgamma_1\tgamma_2^{-1}\in W_{p-j,p-j}$,
so we in fact have
\begin{eqnarray*}
\| S_j^\zeta \|^2 & \leq & \displaystyle\sum_{\substack{
\tgamma_1,\tgamma_2 \in W_{p-j,0} \\
\tgamma_1\tgamma_2^{-1}\in W_{p-j,p-j}
}}
\left( \displaystyle\sum_{\substack{
\xi\in W_{j,0}\\
\tgamma_1 \xi, \tgamma_2 \xi\in\Wp
}}
f_3(\tgamma_1 \xi)f_3(\tgamma_2 \xi)  \right)^2 .
\end{eqnarray*}
Applying the Cauchy-Schwarz inequality we thus obtain
\begin{eqnarray*}
\| S_j^\zeta \|^2 & \leq & \displaystyle\sum_{\substack{
\tgamma_1,\tgamma_2 \in W_{p-j,0}\\
\tgamma_1\tgamma_2^{-1}\in W_{p-j,p-j}
}}
\left( \displaystyle\sum_{\substack{
\xi_1\in W_{j,0}\\
\xi_1\tgamma_1\in\Wp
}}
f_3(\tgamma_1 \xi_1)^2 \right)
\left( \displaystyle\sum_{\substack{
\xi_2\in W_{j,0}\\
\tgamma_2 \xi_2\in\Wp
}}
f_3(\tgamma_2 \xi_2)^2 \right) \\
& \leq & \left( \displaystyle\sum_{\gamma_1\in\cG} f_3(\gamma_1)^2\right)
\left( \displaystyle\sum_{\gamma_2\in\cG} f_3(\gamma_2)^2\right)  \\
& = & \| f_3 \|_2^4
\end{eqnarray*}
as desired.
\end{proof}

This concludes our proof of Theorem~\ref{gpd-thm} in the~$\tA_2$ case.
\hfill $\qedsymbol$

\section{Simply transitive actions}\label{simply transitive actions}

Suppose that~$\G\leq\shaut$ acts simply transitively on~$\vD$.
Fix any vertex~$v_0\in\vD$ and let
\[
N=\{ a\in\G\st d(v_0,av_0)=1\}.
\]
The {\bf Cayley graph} of~$\G$ constructed via right multiplication with
respect to the set~$N$ has~$\G$ itself as its vertex set and has
$\{(c,ca)\st c\in\G,a\in N\}$ as its edge set.
There is a natural action of~$\G$ on its Cayley graph via left multiplication.
Using the convention that an undirected edge between vertices~$u$ and~$v$ in
a graph represents the pair of directed edges~$(u,v)$ and~$(v,u)$, it is
immediate that the $\G$-map $c\mapsto cv_0$ from~$\G$ to~$\D$ is an
isomorphism between the Cayley graph of~$\G$ and the one-skeleton of~$\D$.
In this way we identify~$\G$ with the vertex set~$\vD$ of the building.
Connectivity of the building implies that~$N$ is a generating set for~$\G$.

It is traditional to label the directed edge~$(c,ca)$ with the
generator~$a\in N$. More generally, to the pair $(c,d)\in\G\times\G$ we
assign the
label~$c^{-1}d$. Equivalently, to the pair~$(c,cd)$ we assign the
label~$d\in\G$. Suppose this label is written as a product of generators;
$d=a_1\cdots a_j$. Then there is a path $(c,ca_1,ca_1a_2,\ldots,cd)$
from~$c$ to~$cd$ whose successive edges are labelled $a_1,\ldots,a_j$.
The left translate of~$(c,cd)$ by $b\in\G$ is~$(bc,bcd)$ and also carries
the label~$d$. Conversely, any pair $(c',c'd)$ which carries the label~$d$
is the left translate by~$b=c'c^{-1}$ of~$(c,cd)$. Thus two pairs carry
the same label if and only if one is the left translate of the other.

We define a shape function on~$\G$ by
\[
\sh(b)=\sh(v_0,bv_0)
\]
for $b\in\G$. The pair~$(c,cb)$ has label~$b$ and its shape, defined via
the identification of the Cayley graph and the one-skeleton of~$\D$, is
\[
\sh(c,cb)=\sh(cv_0,cbv_0)=\sh(v_0,bv_0)=\sh(b).
\]
A different choice of~$v_0$ leads to a shape function on~$\G$ which
differs from the first by precomposition with an inner automorphism
of~$\G$.

Suppose $\D$ is an $\tA_{n_1}\times\cdots\times\tA_{n_k}$~building and
that $\G\leq\autD$ consists of type-rotating automorphisms and acts
simply transitively on $\vD$.  Then $\G$~is called an
{\bf {\boldmath$\tA_{n_1}\times\cdots\times\tA_{n_k}$} group}.

\subsection{The case of simply transitive group actions on the
vertices of~{\boldmath{$\tA_1\times\tA_1$}} buildings.}
\label{simply transitive A1}

Suppose that $\D=\D_1\times\D_2$ is an $\tA_1\times\tA_1$ building and
$a\in\shaut$. If we write $(u,v)$ for a generic vertex of
$\D=\D_1\times\D_2$ we have
\[
a(u,v)=(a_1 u,a_2 v)
\]
for some type rotating automorphisms $a_i$ of $\D_i$.  Indeed, suppose
that $(u,v)$ and $(u,v')$ are neighbouring vertices in~$\D$.
Then~$a(u,v)=(x,y)$ and $a(u,v')=(x',y')$ are neighbouring vertices
in~$\D$ and the type-rotating assumption on~$a$ means that
$\tau(x)=\tau(x')$.  Since neighbouring vertices in~$\D_1$ have
distinct types we must have $x=x'$. By induction on~$d(v,v')$, we see
that the first coordinate of~$a(u,v)$ is independent
of~$v\in\D_2$. Similarly, the second coordinate of~$a(u,v)$ is
independent of~$u\in\D_1$. Thus there exist maps~$a_1$ of~$\D_1$ and
$a_2$ of~$\D_2$ such that~$a(u,v)=(a_1 u,a_2 v)$. Since~$a\in\shaut$
it follows that~$a_i$ is a (type-rotating) automorphism of~$\D_i$. Thus
each~$a\in\shaut$ acts as~$a_1\times a_2$ for some (type-rotating)
automorphisms~$a_i$ of~$\D_i$.

Consider an $\tA_1\times\tA_1$ group, that is a group $\G\leq\shaut$ which acts
simply transitively on~$\vD$. A simple example of such a group
is $\FF_p\times\FF_q$ where $\FF_j$ denotes the free group on~$j$ generators.
More generally, $\G=\G_1\times\G_2$ is an~$\tA_1\times\tA_1$ group whenever
$\G_1$ and $\G_2$ are $\tA_1$ groups. However there are also many examples of
$\tA_1\times\tA_1$ groups generated by $\tA_1$ groups $\G_1$ and $\G_2$
where the elements of~$\G_1$ do not all commute with those of~$\G_2$.
 In fact S. Mozes \cite[Theorem 3.2]{Mo} (see also~\cite{bm}) has given
examples of
$\tA_1$ groups $\G_1$, $\G_2$ acting on trees $T_1$, $T_2$ which are embedded
``diagonally'' in $\aut(T_1 \times T_2)$, so that the embeddings do not commute
but the group generated by $\G_1 \cup \G_2$ is an $\tA_1\times\tA_1$ group
acting on
$T_1 \times T_2$.
In other words, even though each~$a\in\G$ is a direct
product of automorphisms,~$\G$ is not a direct product
of the groups $\G_1$ and $\G_2$ .

As
above, fix a vertex $v_0=(v_1,v_2)\in\vD$ and
suppose that~$b=b_1\times b_2\in\G$. Recall that
\[
\sh(b)=\sh(v_0,bv_0)= \left( d(v_1,b_1v_1), d(v_2,b_2v_2)\right),
\]
so that
\[
\sh(b_1\times b_2)=\left(\sh_1(b_1),\sh_2(b_2)\right)
\]
where $\sh_i$ is the shape function on $\aut(\D_i)$.
Consider the generating set
\[
N=\{ a\in\G\st d(v_0,av_0)=1\}
\]
of~$\G$. Let
\[
N_1=\{ a\in\G\st \sh(a)=(1,0)\}
\quad \text{ and }\quad
N_2=\{ a\in\G\st \sh(a)=(0,1)\} .
\]
Then each element~$c\in\G$ has a unique reduced expression of the form
\[
c=a_1\cdots a_mb_{m+1}\cdots b_{m+n}
\]
for some $a_i\in N_1$ and $b_i\in N_2$ and the shape function satisfies
\[
\sh( a_1\cdots a_mb_{m+1}\cdots b_{m+n})=(m,n) .
\]

\begin{cor}\label{A1 group case}
If $\G$ is an~$\tA_1\times\tA_1$ group, $\rho$
is the right regular representation of $\ell^1(\G)$ on $\ell^2(\G)$ defined by
$\rho(g)f=f*g$, and $g\in\CC\G$ is supported on words of shape $(m,n)$ then
\[
\|\rho(g)\|\leq (m+1)(n+1)\;\|g\|_2.
\]
Consequently~$\G$ has property~(RD).
\end{cor}
\begin{proof}
These  results follow from Theorem~\ref{gp-cor}
and Corollary~\ref{RD-cor}.
\end{proof}

\begin{remark}
Since~$\ZZ^2$ and the free group~$\FF_2$ can both be realized as subgroups of
$\tA_1\times\tA_1$ groups, our approach provides a unified method to establish
property~(RD) for these groups.
This answers a question in~\cite[Section~6]{Ha2}.
\end{remark}

\subsection{The case of simply transitive group actions on the vertices
of~{\boldmath{$\tA_2$}} buildings.}

Suppose now that $\D$~is an $\tA_2$~building and $\G\leq\shaut$~acts
simply transitively on~$\vD$.  Recall that such a group~$\G$ is called
an $\tA_2$~group. A detailed account of
$\tA_2$~groups may be found in \cite{CMSZ}, while \cite{RR} and
\cite{RS} have quick introductions to $\tA_2$~buildings and groups.
Some, but not all, $\tA_2$~groups can be embedded as lattices in
matrix groups of the form $PGL_3(K)$ where $K$~is a local field
\cite[II \S8]{CMSZ}.

For an~$\tA_2$ group~$\G$, the shape function~$\sh$ defined
in~\S\ref{simply transitive actions}  has an illuminating algebraic
interpretation. Recall that for any fixed vertex~$v_0\in\vD$ the set
\[
N=\{ a\in\G \st \sh(v_0,av_0)=(1,0) \}
\]
is a set of generators of~$\G$.  In terms of these generators $\G$~has
a presentation of the form
\[
\G=\langle a_i \st a_xa_ya_z=1 \text{ for } (x,y,z)\in\cT\rangle
\]
where $\cT$ is a so-called triangle presentation.
Each element of~$\G$ has a unique shortest expression of the form
\[
x=a_{i_1}\cdots a_{i_m}a^{-1}_{i_{m+1}}\cdots a^{-1}_{i_{m+n}}
\]
where the $a_{i_j}\in N$ and the shape function~$\sh$ on~$\G$ is given
by
\[
\sh(a_{i_1}\cdots a_{i_m}a^{-1}_{i_{m+1}}\cdots a^{-1}_{i_{m+n}}) = (m,n).
\]

\begin{cor}\label{A2 group case}
Let $\G$ be an~$\tA_2$ group and $\rho$ be the right regular representation of
$\ell^1(\G)$ on $\ell^2(\G)$ defined by $\rho(g)f=f*g$. If $g\in\CC\G$ is
supported
on words of shape $(m,n)$ then
\[
\|\rho(g)\|\leq
1/2(m+1)(n+1)(m+n+2)\sqrt{\max(m,n)+1}\;\|g\|_2.
\]
Moreover,~$\G$ has property~(RD).
\end{cor}
\begin{proof}
These results follow from Theorem~\ref{gp-cor}
and Corollary~\ref{RD-cor}.
\end{proof}

\end{document}